\documentclass[a4paper,11pt]{amsart}
\usepackage[margin=3cm]{geometry}

\usepackage[T1]{fontenc}
\usepackage{lmodern, amsfonts,amsmath,amstext,amsbsy,amssymb,
amsopn,amsthm,upref,eucal,mathptmx,mathtools,url}

\usepackage{mathrsfs}

\RequirePackage{xcolor} 
\definecolor{halfgray}
{gray}{0.55}
\definecolor{webgreen}
{rgb}{0,0.4,0}
\definecolor{webbrown}
{rgb}{.8,0.1,0.1}
\definecolor{red}
{rgb}{1,0,0}
\usepackage{microtype}
\usepackage{tikz}

\newcommand \R {{ \mathbb R}}

\newcommand \Hbb {{ \mathbb H}}

\newcommand \Z {{ \mathbb Z}}
\newcommand \N {{ \mathbb N}}

\newcommand {\cN} {{\mathcal N}}

\newcommand*{\diff}{\mathop{}\!\mathrm{d}}

\newcommand{\norm}[1]{\left\lVert#1\right\rVert}

\newcommand{\Dp}[1]{\mathcal{D}_\mu^{+} #1}
\newcommand{\Dm}[1]{\mathcal{D}_\mu^{-} #1}

\newcommand{\Dpm}[1]{\mathcal{D}_\mu^{\pm} #1}
\newcommand{\Dpmp}[1]{\mathcal{D}_\mu^{\pm}( #1)}
\newcommand{\RxT}{\mathcal{R}_\mu f(x,T)}
\newcommand{\aver}[1]{A_{#1}(x,T)}

\newcommand{\SL}{%
\operatorname{SL}
}
\newcommand{\PSL}{%
\operatorname{PSL}
}
\newcommand{\PSO}{%
\operatorname{PSO}
}

\newcommand{\Levy}{d_{\levyy}}

\DeclareMathOperator{\vol}{vol}
\DeclareMathOperator{\un}{u}
\DeclareMathOperator{\emb}{emb}
\DeclareMathOperator{\Spec}{Spec}

\DeclareMathOperator{\comp}{comp}
\DeclareMathOperator{\Id}{Id}
\DeclareMathOperator{\princ}{princ}
\DeclareMathOperator{\FF}{FF}
\DeclareMathOperator{\levyy}{LP}

\newtheorem{theorem}{Theorem}
\newtheorem {lemma}[theorem]{Lemma}
\newtheorem {proposition}[theorem]{Proposition}

\newtheorem{remark}[theorem]{Remark}

\date{\today}

\author[Davide Ravotti]{Davide Ravotti \vspace{5pt} \\ with an appendix by Emilio Corso}

\address{
Monash University, School of Mathematics \\ Clayton Campus, 3800 Victoria, Australia
}
\address{
Universit\"{a}t Wien, Department of Mathematics \\ Oskar-Morgenstern-Platz 1, 1090 Wien, Austria
}

\email{davide.ravotti@gmail.com}

\address{
ETH Z\"{u}rich, R\"{a}mistrasse 101 \\ CH-8092 Z\"{u}rich, Switzerland
}

\email{emilio.corso@math.ethz.ch}

 \title[Asymptotics and limit theorems for horocycle integrals]
 {Asymptotics and limit theorems for horocycle ergodic integrals \`a la Ratner}

\begin{document}
\maketitle{}

\begin{abstract}
We apply a method inspired by Ratner's work on quantitative mixing for the geodesic flow (Ergod. Theory Dyn. Syst., 1987) and developed by Burger (Duke Math. J., 1990) to study ergodic integrals for horocycle flows. 
We derive an explicit asymptotic expansion for horocycle averages, recovering a celebrated result by Flaminio and Forni (Duke Math. J., 2003), and we show that the coefficients in the asymptotic expansion are H\"{o}lder continuous with respect to the base point.
Furthermore, we provide short and 
streamlined proofs of the spatial limit theorems of Bufetov and Forni (Ann. Sci. Éc. Norm. Supér., 2014) and, in an appendix by Emilio Corso, of a temporal limit theorem by Dolgopyat and Sarig (J. Stat. Phys., 2017).
\end{abstract}

\section{Introduction}\label{sec:introduction}

The prime example of a parabolic homogeneous flow is, arguably, the horocycle flow on finite volume quotients of the Lie group $\PSL(2,\R)$.
On each such quotient, the horocycle flow acts by multiplication by the one-parameter subgroup of upper triangular unipotent matrices.
Geometrically, it can be described as follows. Consider a finite area hyperbolic surface $S$; the geodesic flow $\{ \phi^X_t\}_{t \in \R}$ acts on the unit tangent bundle $T^1S$ of $S$ by moving each unit tangent vector $v \in T^1S$ along the unique geodesic starting at $v$. The \emph{stable leaf} at $v$ is a smooth curve passing through $v$ which consists of all tangent vectors $w \in T^1S$ whose images $\phi^X_t(w)$ get exponentially close to $\phi^X_t(v)$ as $t \to +\infty$. The horocycle flow $\{h_t\}_{t \in \R}$ moves the point $v$ along its stable leaf at unit speed.

Besides the interest from the parabolic dynamics perspective, the study of the horocycle flow (and, more in general, of unipotent actions) has lead to important breakthroughs in other areas of mathematics, notably in number theory and in mathematical physics.
Its dynamical and ergodic properties are now well-understood thanks to the work of several authors: interestingly, the horocycle flow displays an \lq\lq intermediate\rq\rq\ chaotic behaviour, with some features which are usually associated to orderly ergodic systems (zero entropy \cite{Gur}, minimality \cite{Hed}, unique ergodicity \cite{Fur}) and others to highly chaotic ones (mixing of all orders \cite{Mar}, Lebesgue spectrum \cite{Par}). This is a typical trait of the parabolic paradigm; we refer the reader to \cite[Chapter 8]{HasKat} and to the introduction of \cite{AFRU} for an extensive discussion on this topic.

In this paper, we will be interested in the ergodic properties of the horocycle flow when the phase space $M = \Gamma \backslash \PSL(2,\R)$ is \emph{compact}.
It is a classical result of Furstenberg \cite{Fur} that the horocycle flow on $M$ is uniquely ergodic; that is, the Haar measure on $M$ is the unique invariant probability measure. 
With respect to this measure, it is mixing of all orders \cite{Mar} and has Lebesgue spectrum \cite{GeFo, Par}. 
Some finer properties, including some remarkable rigidity results, were studied by Ratner \cite{Rat2, Rat4, Rat3}.

Quantitative statements concerning the ergodic and mixing properties of the horocycle flow are of great importance, in particular for some applications such as for the study of its time-changes.
In \cite{Rat1}, Ratner showed that the rate of the decay of correlations for H\"{o}lder observables is \emph{polynomial}, and the optimal exponent depends on the \emph{spectral gap}, that is, on the smallest positive eigenvalue of the Laplace-Beltrami operator on the underlying hyperbolic surface. The rate of equidistribution of horocycle translates of generic arcs, from which one can deduce Ratner's estimates, was established in \cite{Rav}, as a consequence of the work of Bufetov and Forni \cite{BuFo}, which we recall below.

Effective statements on the equidistribution of orbits are now well-known as well. 
Burger proved a polynomial bound on the ergodic integrals for sufficiently smooth functions \cite{Bur}, where the exponent depends again on the spectral gap, and is \emph{half} of the mixing exponent. 
This result was improved in the seminal work of Flaminio and Forni \cite{FlaFo}, who established a full asymptotic expansion for the ergodic integrals of the horocycle flow in terms of the horocycle invariant distributions.
Bufetov and Forni \cite{BuFo} later refined Flaminio and Forni's theorem by constructing H\"{o}lder functionals which govern the asymptotics of horocycle integrals. 
Similar objects were first introduced and studied by Bufetov for translation flows \cite{Buf}, and are sometimes called \emph{Bufetov functionals}. 
In the same work \cite{BuFo}, the authors derived spatial limit theorems for a large class of functions (namely, functions which are not fully supported on the discrete series).
Notably, Bufetov and Forni's results show that the limiting distributions, when they exist, are non degenerate and \emph{compactly supported}. This is in sharp contrast with, for example, the case of the geodesic flow and of many other hyperbolic systems, for which a Central Limit Theorem holds.

%
In this paper, we study the ergodic integrals for horocycle flows, following a method inspired by Ratner's beautiful paper \cite{Rat1} and further developed by Burger \cite{Bur}. This approach has the advantage of being rather simple and, furthermore, of providing \emph{explicit} formulas. A similar strategy has been employed by Str\"{o}mbergsson \cite{Str} to derive effective bounds on horocycle integrals on finite volume, noncompact surfaces, and by Edwards \cite{Edw} to study the equidistribution rates of translates of pieces of horospheres in quotients of semisimple groups. In our setting, we derive much more precise results, as we now describe. 

We establish a full asymptotic expansion for horocycle integrals, thus providing a short proof of Flaminio and Forni's result which does not rely on the study of the cohomological equation and on the classification of invariant distributions.
We can also show that the coefficients in the expansion are H\"{o}lder continuous, a property which, to the best of our knowledge, was not known before. 
We then recover the limit theorems of Bufetov and Forni, and we obtain explicit expressions for the limiting distributions.
Finally, in an appendix by Emilio Corso, we derive a short proof of the temporal distributional limit theorem by Dolgopyat and Sarig \cite{Dolgopyat-Sarig}.

In a recent work \cite{AdBa}, Adam and Baladi study horocycle flows in a general setting that includes the case of surfaces of variable negative curvature, and they establish power law convergence for the ergodic integrals. We believe that it would be interesting to compare the strategies of the present paper and of their work, which relies on the spectral theory of transfer operators.

\subsection{Organisation of the paper}

In Section \ref{sec:statement}, we state the main result of the paper, Theorem \ref{thm:main1}, and its consequences: Flaminio and Forni's asymptotic expansion (Theorem \ref{thm:main_FF}), Bufetov and Forni's spatial limit theorem (Theorem \ref{thm:main_BF1}), and Dolgopyat and Sarig's temporal limit theorem (Theorem \ref{thm:main_DS}). 
Section \ref{sec:reduction} contains the key idea, Proposition \ref{thm:eq_diff}, derived from Ratner's and Burger's works, which reduces the problem to solving a system of linear ordinary differential equations. In Sections \ref{section:positive_param} and \ref{section:discrete_ser}, we write the solutions of the equations and, in doing so, we prove Theorem \ref{thm:main1}. Finally, the proofs of Theorems \ref{thm:main_FF}, \ref{thm:main_BF1}  and, \ref{thm:main_DS} are contained in \S\ref{section:FF}, 
\S\ref{sec:limit}, and in the Appendix \S\ref{sec:temporal_lt}, respectively.

\subsection{Acknowledgements}

I would like to thank Livio Flaminio, Giovanni Forni, Omri Sarig for several enlightening discussions, and Henk Bruin and Raphael Steiner for their comments on an earlier draft.
This work was partially supported by the Australian Research Council.

\section{Statement of the main results}\label{sec:statement}

\subsection{Preliminaries and notation}

We denote by $\PSL(2,\R)$ the group of $2 \times 2$ real matrices with determinant 1 quotiented by $\{\pm I_2\}$. 
By a little abuse of notation, we write elements of $\PSL(2,\R)$ as matrices in $\SL(2,\R)$.
The Lie algebra $\mathfrak{sl}_2(\R)$ of $\PSL(2,\R)$ is the 3-dimensional real vector space of matrices of the same size with zero trace.
We fix the basis $\{ U,X,V \}$ of $\mathfrak{sl}_2(\R)$ given by
$$
U =
\begin{pmatrix}
0 & 1 \\
0 & 0
\end{pmatrix},
\ \ \ 
X =
\begin{pmatrix}
1/2 & 0 \\
0 & -1/2
\end{pmatrix},
\ \ \ 
V =
\begin{pmatrix}
0 & 0 \\
1 & 0
\end{pmatrix},
$$
for which the following commuting relations hold:
$$
[X,U] = U, \quad [X,V] = -V, \quad [U,V]= 2X.
$$
The elements $U, X, V$ generate, by exponentiating, the one-parameter subgroups of upper triangular unipotent, diagonal, and lower triangular unipotent matrices, respectively. 

Let us fix a uniform lattice $\Gamma < \PSL(2,\R)$, namely $\Gamma$ is a discrete subgroup of $\PSL(2,\R)$ such that the quotient $M := \Gamma \backslash \PSL(2,\R)$ is compact. 
We will denote by $\vol$ the unique $\PSL(2,\R)$-invariant probability measure on $M$, namely the probability measure locally given by the Haar measure.
The manifold $M$ can be identified with the unit tangent bundle of the compact hyperbolic surface $S = \Gamma \backslash \Hbb$, where $\Gamma$ acts on the upper half plane $\Hbb$ by M\"{o}bius transformations (recall that $\Hbb = \PSL(2,\R) / \PSO(2)$, where the group $\PSO(2)$ of rotation matrices is the stabilizer of the point $i \in \Hbb$ under the transitive action of $\PSL(2,\R)$ on $\Hbb$ by M\"{o}bius transformations).

The group $\PSL(2,\R)$ acts on $M$ by right multiplication. 
The restrictions of this action to the one-parameter subgroups generated by $U, X$ and $V$ are, respectively, the \emph{stable horocycle flow} $\{h_t\}_{t \in \R}$, the \emph{geodesic flow} $\{\phi^X_t\}_{t \in \R}$, and the \emph{unstable horocycle flow} $\{h^{\un}_t\}_{t \in \R}$.
Explicitly, they are given by
$$
h_t(\Gamma g) = \Gamma g 
\begin{pmatrix}
1 & t \\
0 & 1
\end{pmatrix},
\quad 
\phi^X_t(\Gamma g) = \Gamma g 
\begin{pmatrix}
e^{t/2} & 0 \\
0 & e^{-t/2}
\end{pmatrix},
\quad \text{ and } \quad
h^{\un}_t(\Gamma g) = \Gamma g 
\begin{pmatrix}
1 & 0 \\
t & 1
\end{pmatrix}.
$$
for all $g \in \PSL(2,\R)$ and $t \in \R$.

Let $f \in \mathscr{C}^2(M)$ be a twice differentiable function on $M$. 
We are interested in studying the asymptotics of the horocycle ergodic averages of $f$, namely of 
$$
\aver{f} := \frac{1}{T} \int_0^T f \circ h_t (x) \diff t,
$$
defined for every $x \in M$ and $T \geq 1$, as $T \to \infty$.
We recall that the second order differential operator 
$$
\square = -X^2+X-UV,
$$
called the \emph{Casimir operator}, is a generator of the centre of the universal enveloping algebra of $\mathfrak{sl}_2(\R)$, and hence commutes with $U, X, V$, and the associated homogeneous flows. 
It acts as an essentially self-adjoint operator on $L^2(M)$, in particular its eigenvalues are real.
If $\mu \in\R$ is an eigenvalue of $\square$, let $\nu \in \R_{\geq 0 } \cup i \R_{>0}$ be such that
$$
\frac{1-\nu^2}{4} = \mu.
$$

\subsection{The main result}

The main result of this paper is the following theorem, from which we will deduce several well-known results on the asymptotic behaviour of ergodic averages and limit theorems for horocycle flows.
We stress that the proof of Theorem \ref{thm:main1} is \emph{explicit}, and the terms appearing in the statement are defined in Sections \ref{section:positive_param} and \ref{section:discrete_ser} below. 

\begin{theorem}\label{thm:main1}
Let $f \in \mathscr{C}^2(M)$ be an eigenfunction of the Casimir operator with eigenvalue $\mu \in \R$.
\begin{itemize}
\item[(i)] If $\mu > 1/4$, there exist two H\"{o}lder continuous functions $\Dp{f}$, $\Dm{f}$, with H\"{o}lder exponent $1/2$ and with 
$$
\| \Dpm{f} \|_{\infty} \leq \left( \frac{11}{\Im \nu} + 1 \right) \|f\|_{\mathscr{C}^2},
$$
such that for all $x \in M$ and $T \geq 1$ we have
\begin{equation*}
\begin{split}
\aver{f} =& T^{-\frac{1}{2}} \cos\left(\frac{\Im \nu}{2} \log T \right) \, \Dp{ f(\phi^X_{\log T}(x))} + T^{-\frac{1}{2}} \sin\left(\frac{\Im \nu}{2} \log T \right) \, \Dm{ f(\phi^X_{\log T}(x))} \\
&+ \RxT,
\end{split}
\end{equation*}
where the remainder term $\RxT$ satisfies
$$
|\RxT| \leq \frac{16}{\Im \nu} \|f\|_{\mathscr{C}^2} T^{-1}.
$$
Moreover,
\begin{equation}\label{eq:aux_thm_1_princ}
| \aver{f}| \leq \frac{15(\log T +1)}{\sqrt{T}} \|f\|_{\mathscr{C}^2}.
\end{equation}
\item[(ii)] If $\mu = 1/4$, there exist two H\"{o}lder continuous functions, $\mathcal{D}^{+}_{1/4}f$ with H\"{o}lder exponent $1/2 - \varepsilon$ for all $\varepsilon >0$, and $\mathcal{D}^{-}_{1/4}f$ with H\"{o}lder exponent $1/2$, and with 
$$
\| \mathcal{D}^{\pm}_{1/4}f \|_{\infty} \leq 9 \|f\|_{\mathscr{C}^2},
$$
such that for all $x \in M$ and $T \geq 1$ we have
$$
\aver{f} = T^{-\frac{1}{2}} \, \mathcal{D}^{+}_{1/4} f(\phi^X_{\log T}(x)) + T^{-\frac{1}{2}}  \log T \, \mathcal{D}^{-}_{1/4} f(\phi^X_{\log T}(x)) +\mathcal{R}_{1/4}f(x,T),
$$
where the remainder term $\mathcal{R}_{1/4}f(x,T)$ satisfies
$$
|\mathcal{R}_{1/4}f(x,T)| \leq 8 \frac{\log T +2}{T} \|f\|_{\mathscr{C}^2}.
$$
\item[(iii)] If $0< \mu < 1/4$, there exist two H\"{o}lder continuous functions $\Dp{f}$, $\Dm{f}$, with H\"{o}lder exponents $\frac{1-\nu}{2} $ and $\frac{1+\nu}{2} $ respectively, and with 
$$
\| \Dpm{f} \|_{\infty} \leq \frac{6}{\nu(1-\nu)} \|f\|_{\mathscr{C}^2},
$$
such that for all $x \in M$ and $T \geq 1$ we have
$$
\aver{f} = T^{-\frac{1+\nu}{2}} \, \Dp{f(\phi^X_{\log T}(x)) } + T^{-\frac{1-\nu}{2}} \, \Dm{ f(\phi^X_{\log T}(x))} + \RxT,
$$
where the remainder term $\RxT$ satisfies
$$
|\RxT| \leq \frac{8}{(1-\nu^2) \nu} \|f\|_{\mathscr{C}^2} T^{-1}.
$$
Moreover,
\begin{equation}\label{eq:aux_thm_1_comp}
| \aver{f}| \leq \frac{15}{(1-\nu)^2} \|f\|_{\mathscr{C}^2} (\log T +1) T^{-\frac{1-\nu}{2}}
\end{equation}
\item[(iv)] If $\mu = 0$, then we have
$$
\aver{f} = \vol(f) + \frac{1}{T} \int_0^{\log T} \left( Vf \circ \phi^X_\xi \circ h_T(x) - Vf\circ \phi^X_\xi (x) \right) \diff \xi + \mathcal{R}_{0}f(x,T),
$$
where the remainder term $\mathcal{R}_{0}f(x,T)$ satisfies
$$
|\mathcal{R}_{0}f(x,T)| \leq \frac{3}{T} \|f\|_{\mathscr{C}^2}.
$$
\item[(v)] If $\mu < 0$, then we have
$$
| \aver{f} | \leq \frac{5}{T} \|f\|_{\mathscr{C}^2};
$$
in particular, $f$ is a continuous coboundary, namely there exists a continuous function $u$ such that $f=Uu$.
\end{itemize}
\end{theorem}

\subsection{Flaminio and Forni's Theorem}

From Theorem \ref{thm:main1}, using some basic facts from harmonic analysis, we can recover the seminal result of Flaminio and Forni on horocycle ergodic averages \cite[Theorem 1.5]{FlaFo}. 
Moreover, we strengthen their theorem by showing that the coefficients in the asymptotic expansion are H\"{o}lder continuous with respect to the base point.

In order to state our result, let us introduce some further notation.
Let 
$$
Y =
\begin{pmatrix}
0 & -1/2 \\
-1/2 & 0
\end{pmatrix},
\ \ \ \text{ and }\ \ \ 
\Theta =
\begin{pmatrix}
0 & 1/2 \\
-1/2 & 0
\end{pmatrix},
$$
and define 
$$
\Delta = -(X^2 + Y^2 + \Theta^2) = \square - 2 \Theta^2.
$$
The operator $\Delta$ acts as an essentially self-adjoint \emph{elliptic} operator on $L^2(M)$, namely is a \emph{Laplacian} on $M$. 
We remark that $\Delta$ and $\square$ act as the Laplace-Beltrami operator on $L^2(S)$, where, we recall, $S = \Gamma \backslash \Hbb = \Gamma \backslash \PSL(2,\R) / \PSO(2)$.

For every $r >0$, we denote by $W^r=W^r(M)$ the Sobolev space of functions $f \in L^2(M)$ such that $\Delta^{r/2}f \in L^2(M)$, namely $W^r(M)$ is the maximal domain of the operator $(\Id + \Delta)^{r/2}$ with the inner product
\begin{equation}\label{eq:def_inn_prod_Wr}
\langle f,g \rangle_{W^r} = \langle (\Id + \Delta)^rf,g \rangle,
\end{equation}
where $\langle \cdot, \cdot \rangle$ is the inner product in $L^2(M)$.
The space $W^r(M)$ coincides with the closure of $\mathscr{C}^\infty(M)$ with respect to the norm $\| \cdot \|_{W^r}$ induced by the inner product above.

By the Sobolev Embedding Theorem, $W^4(M) \subset \mathscr{C}^2(M)$ and there exists a constant $C_{\emb}>0$ such that 
$$
\| f\|_{\mathscr{C}^2} \leq C_{\emb} \|f\|_{W^4},
$$
for all $f \in W^4(M)$. Clearly, we could replace $W^4(M)$ with any $W^r(M)$, provided that $r > 7/2$.

We denote by $\Spec(\square)$ the spectrum of the Casimir operator $\square$ on $L^2(M)$. 
Since $M$ is compact, it is well-known that $\Spec(\square)$ is discrete; moreover, $\Spec(\square) \cap \R_{\geq 0}$ coincides with the spectrum of the Laplace-Beltrami operator on the hyperbolic surface $S$. 
We call 
$$
\sigma_{\comp} = \Spec(\square) \cap (0,1/4), \text{\ \ \ and\ \ \ }\sigma_{\princ} = \Spec(\square) \cap (1/4, \infty),
$$
and we let
$$
\varepsilon_0 = 
\begin{cases}
1 & \text{if } 1/4 \in \Spec(\square),\\
0 & \text{otherwise}.
\end{cases}
$$
Our version of Flaminio and Forni's Theorem is the following.
\begin{theorem}\label{thm:main_FF}
There exists a constant $C_M$, defined explicitly in \eqref{eq:defin_CM}, such that the following holds. Let $f \in W^6(M)$. For all $\mu \in \Spec(\square) \cap \R_{>0}$, there exist bounded functions $\Dp{f}, \Dm{f}$ satisfying 
$$
\sum_{\mu \in \Spec(\square) \cap \R_{>0}} \| \Dpm{f}\|_{\infty} \leq  C_M \|f \|_{W^6},
$$
for which the following holds. For all $x \in M$ and $T \geq 1$, there exists $\mathcal{R}f(x,T)$, with
$$
|\mathcal{R}f(x,T)| \leq  C_M \|f \|_{W^6} \frac{1+ \log T}{T},
$$
such that  
\begin{equation*}
\begin{split}
\frac{1}{T} \int_0^T f \circ h_t(x) \diff t = &\int_M f \diff \vol + \sum_{\mu \in \sigma_{\comp}} \left( T^{-\frac{1+\nu}{2}} \, \Dp{f(x_T) } + T^{-\frac{1-\nu}{2}} \, \Dm{ f(x_T)} \right)\\
&+ \sum_{\mu \in \sigma_{\princ}}\left( T^{-\frac{1}{2}} \cos\left(\frac{\Im\nu}{2} \log T \right) \, \Dp{ f(x_T)} + T^{-\frac{1}{2}} \sin\left(\frac{\Im\nu}{2} \log T \right) \, \Dm{ f(x_T)} \right) \\
&+ \varepsilon_0 \cdot \left( T^{-\frac{1}{2}} \, \mathcal{D}^{+}_{1/4} f(x_T) + T^{-\frac{1}{2}} \log T \, \mathcal{D}^{-}_{1/4} f(x_T) \right) + \mathcal{R}f(x,T),
\end{split}
\end{equation*}
where $x_T= \phi^X_{\log T}(x)$. In fact, the functions $\Dpm{f}$ are \emph{H\"{o}lder continuous} with exponent $\frac{1 \mp \Re\nu}{2}$, apart from $\mathcal{D}^{+}_{1/4}f$ which has exponent $\frac{1}{2}- \varepsilon$, for all $\varepsilon >0$.
\end{theorem}

The proof of Theorem \ref{thm:main_FF} is contained in Section \ref{section:FF} and follows from Theorem \ref{thm:main1} by exploiting a standard decomposition of the Sobolev space $W^6(M)$ into a direct sum of irreducible subspaces.

Let us further comment on the relation to \cite[Theorem 1.5]{FlaFo}.
Fix $\mu \in \sigma_{\comp} = \Spec(\square) \cap (0,1/4)$.
By comparison with Flaminio and Forni's result we can write 
$$
\Dpm{f}(x_T) = c(x,T) \mathcal{D}^{\FF, \pm}_\mu (f),
$$
where $\mathcal{D}^{\FF, \pm}_\mu$ are the horocycle invariant distributions classified by Flaminio and Forni in \cite{FlaFo}.
An analogous equality holds also for $\mu = 1/4$ and for $\mu \in \sigma_{\princ}$, after performing a change of basis (see Proposition \ref{thm:eigenvectors} below).
We note that it follows immediately from Theorem \ref{thm:main_FF} that 
$$
\Dpmp{Uf} =0.
$$
Flaminio and Forni proved that the horocycle invariant distributions are also eigenvectors for the action of the geodesic flow. By straightforward calculations, we can verify that this is the case for the coefficients $\Dpm{f}$ in Theorem \ref{thm:main_FF}. 

\begin{proposition}\label{thm:eigenvectors}
The following identities hold:
\begin{itemize}
\item for $0<\mu< 1/4$:
$$
\Dpmp{Xf} = \frac{1\pm \nu}{2} \Dpm{f},
$$
\item for $\mu=1/4$:
$$
\begin{pmatrix}
\mathcal{D}^{+}_{1/4}(Xf) \\
\mathcal{D}^{-}_{1/4}(Xf)
\end{pmatrix}
= 
\begin{pmatrix}
\frac{1}{2} & -1 \\
0 & \frac{1}{2}
\end{pmatrix}
\begin{pmatrix}
\mathcal{D}^{+}_{1/4}f \\
\mathcal{D}^{-}_{1/4}f
\end{pmatrix},
$$
\item for $\mu > 1/4$:
$$
\begin{pmatrix}
\Dp{(Xf)} \\
\Dm{(Xf)}
\end{pmatrix}
= 
\begin{pmatrix}
\frac{1}{2} & -\frac{\Im \nu}{2} \\
\frac{\Im \nu}{2} & \frac{1}{2}
\end{pmatrix}
\begin{pmatrix}
\Dp{f} \\
\Dm{f}
\end{pmatrix}.
$$
\end{itemize}
\end{proposition}
The proof of Proposition \ref{thm:eigenvectors} is contained in \S\ref{section:geo_action}.

\subsection{Limit Theorems: Bufetov and Forni's Theorem}

From Theorem \ref{thm:main_FF} we can deduce the limit theorems for horocycle integrals which were first established by Bufetov and Forni \cite{BuFo}.
Let us consider a real-valued function $f \in W^6(M)$ with zero average, and let $\Dpm{f}$ be the continuous functions given by Theorem \ref{thm:main_FF}. 
Let us assume that there exists $\mu \in \Spec(\square) \cap \R_{>0}$ for which the function $\Dm{f}$ is not identically zero. In particular, by the Gottschalk-Hedlund Theorem and Theorem \ref{thm:main_FF}, $f$ is not a continuous coboundary (indeed, it follows from the result by Flaminio and Forni \cite{FlaFo} on the cohomological equation that $f$ is not a \emph{measurable} coboundary). 
Let $\mu_f >0$ be the minimum of all such $\mu$'s, and let $\nu_f = \Re \sqrt{1-4\mu_f} \in [0,1)$.
For $T >1$, we denote by $\mathfrak{I}(f,T)$ the distribution of the random variable
\begin{equation*}
\begin{split}
T^{- \frac{1+\nu_f}{2}}\int_0^T f \circ h_t (x) \diff t, &\quad \text{ if } \mu_f \neq \frac{1}{4},\\
(T^{\frac{1}{2}} \log T)^{-1} \int_0^T f \circ h_t (x) \diff t, &\quad \text{ if } \mu_f = \frac{1}{4},
\end{split}
\end{equation*}
where the point $x$ is distributed according to the probability measure $\vol$ on $M$ (we will simply write $x \sim \vol$). 

In order to state our limit theorem, we need to introduce some further notation.
If $\mu_f \leq 1/4$, let $\mathfrak{D}(f)$ be the distribution of the random variable $\mathcal{D}^{-}_{\mu_f}f(x)$, where $x \sim \vol$.
Since, by definition of $\mu_f$, the function $\mathcal{D}^{-}_{\mu_f}f$ is not identically zero and is bounded, the associated probability measure on the real line is not a Dirac mass and it is compactly supported.
If $\mu_f > 1/4$, let $\mathfrak{D}(f,T)$ be the distribution of the following sum of \lq\lq oscillating\rq\rq\ random variables 
$$
\sum_{\mu \in \sigma_{\princ}} \cos\left(\frac{\Im\nu}{2} \log T \right) \, \Dp{ f(x)} + \sin\left(\frac{\Im\nu}{2} \log T \right) \, \Dm{ f(x)},
$$
where $x \sim \vol$.
We denote by $\Levy$ the L\'{e}vy-Prokhorov distance between probability distributions, which induces the topology of weak convergence (see, e.g., \cite{Bil}).
The following distributional limit theorem holds (see \cite[Theorems 1.4, 1.5]{BuFo}).

\begin{theorem}\label{thm:main_BF1} 
Let $f \in W^6(M)$ be real-valued, with $\vol(f)=0$, and let $\mu_f >0 $ be defined as above.
If $\mu_f \leq 1/4$, then there exists an explicit $\eta \in (0,1)$ such that 
$$
\Levy \left( \mathfrak{I}(f,T), \mathfrak{D}(f) \right) \leq 2C_M \|f\|_{W^4} T^{-\eta}(1+\log T).
$$
If $\mu_f=1/4$, then
$$
\Levy \left( \mathfrak{I}(f,T), \mathfrak{D}(f) \right) \leq 2C_M \|f\|_{W^4}(\log T)^{-1}.
$$
If $\mu_f > 1/4$, then
$$
\Levy \left( \mathfrak{I}(f,T), \mathfrak{D}(f,T) \right) \leq C_M \|f\|_{W^4} T^{-\frac{1}{2}}(1+\log T).
$$
\end{theorem}

In particular, the first two cases of Theorem \ref{thm:main_BF1} are classical spatial distributional limit theorems (DLTs): the distribution of the ergodic integrals of $f$, appropriately normalized, converges to a non-atomic, compactly supported distribution.
In the third case, the renormalized ergodic integrals converge to a \lq\lq moving target\rq\rq, that is, to a quasiperiodic motion in the space of random variables. 
It is reasonable to expect that a limit theorem does not hold in this case; however it is still possible that the limiting random variables all have the same distribution. A careful analysis of the formulas established in Section \ref{section:positive_param} might be enough to rule out the possibility of this \lq\lq degenerate\rq\rq\ case, and hence prove the absence of a spatial limit theorem. For the moment, this remains an open problem.

\subsection{Limit Theorems: Dolgopyat and Sarig's Theorem}

Theorem \ref{thm:main_BF1} shows that the standard CLT does not hold for the horocycle flow. As we already remarked, this is in stark contrast with several hyperbolic systems, where the ergodic integrals of sufficiently regular observables satisfy a spatial limit theorem with a Gaussian limit, see, e.g., \cite{Rat:CLT}.

However, a Central Limit Theorem holds when fixing a \emph{deterministic} inital point $x$ and \emph{randomizing} time instead.
Limit theorems for this type of random variables are called \emph{temporal distributional limit theorems}, and have been investigated for several zero entropy dynamical systems; see, e.g., \cite{Bec1, Bec2, ADDS, Dolgopyat-Sarig, PaSo, BrUl} and references therein.
In the setting of this paper, i.e., for horocycle flows on compact surfaces, a temporal DLT was first proved by Dolgopyat and Sarig in \cite{Dolgopyat-Sarig}. They first established the result for \emph{horocycle windings} (namely, for harmonic 1-forms) and then relied on the work of Flaminio and Forni, and of Bufetov and Forni, to deduce it for more general observables. In the Appendix \S\ref{sec:temporal_lt} by Emilio Corso, we provide a direct proof of Dolgopyat and Sarig's temporal DLT; see Theorem \ref{thm:main_DS}, which we now state.

For any $\sigma\in \R$, let $\cN(0,\sigma^{2})$ indicate the Gaussian distribution on the real line with mean $0$ and variance $\sigma^{2}$. For any $T\in \R_{>0}$, let $\mathcal{U}_{[0,T]}$ denote the uniform probability measure, that is, the normalized Lebesgue measure, on the interval $[0,T]\subset \R$.

\begin{theorem}\label{thm:main_DS}
Assume that $1/4 \notin \Spec(\square)$. Let $f \in W^6(M)$ be a real-valued function with $\vol(f)=0$, and assume that $\Dpm{f} \equiv 0$ for all $\mu \in \Spec(\square) \cap \R_{>0}$.
If $f$ is not a measurable coboundary, then there is a real number $\sigma>0$ and, for every $x \in M$, a collection of real numbers $A_T(x) \in \R$ such that  
$$
	\frac{\int_0^{t}f\circ h_s(x) \diff s - A_T(x)}{\sqrt{\log{T}}}\overset{T\to+\infty}{\longrightarrow}\cN(0,\sigma^{2}), \quad t\sim \mathcal{U}_{[0,T]}
$$
in distribution. Therefore, the ergodic integrals of $f$ satisfy a temporal ditributional limit theorem on any horocycle orbit.
\end{theorem}

The constants $A_T(x)$ in Theorem \ref{thm:main_DS} are explicitly defined in \S\ref{sec:temporal_lt}.

\begin{remark}
As we will explain in \S\ref{sec:log_growth} and in the Appendix, the assumptions in Theorem \ref{thm:main_DS} can be replaced by asking that all the components of $f$ corresponding to \emph{positive} Casimir parameters are coboundaries for the horocycle flow, while $f$ itself is not. Under these or under the assumptions of the theorem, the ergodic integrals of $f$ up to time $t$ grow as $\log t$, according to the formula in Theorem \ref{thm:main1}-(iv); we refer the reader to Lemma \ref{lemma:log_growth} and Theorem \ref{temporalDLT} for the details.
\end{remark}

\section{Reduction to a system of ODEs}\label{sec:reduction}

The geodesic and horocycle flows satisfy the well-known commutation relation
$$
\phi^X_t \circ h_s (x) = h_{e^{-t}s} \circ \phi^X_t(x).
$$
In other words, the geodesic flow at time $t$ maps any horocycle orbit segment with unit speed starting at a point $x$ into an horocycle orbit segment at $\phi^X_t(x)$ with constant speed $e^{-t}$.
By a change of variable, we immediately get the following lemma.
\begin{lemma}\label{lemma:c_o_v}
Let $\ell \in \mathscr{C}(M)$. Then, for all $t \geq 0$ we have
$$
\int_0^1 \ell \circ \phi^X_{-t}\circ h_s(x) \diff s = \frac{1}{e^t}\int_0^{e^t} \ell \circ h_s (\phi^X_{-t}(x)) \diff s.
$$
\end{lemma}

Let now $f$ be a $\mathscr{C}^2$ function, and assume that $\square f = \mu f$ for some $\mu \in \R$. 
We define 
$$
J_f(x,t) := \int_0^1 f \circ \phi^X_{-t}\circ h_s(x) \diff s,
$$
so that, by Lemma \ref{lemma:c_o_v}, we have
\begin{equation}\label{eq:aver_equals_I}
\aver{f} = J_f(\phi^X_{\log T}(x), \log T).
\end{equation}
Hence, for any fixed $x \in M$, we now focus our attention on the function $J_f(x, \cdot)$.
The key idea, coming from Ratner's and Burger's works \cite{Rat1, Bur}, is to show that it satisfies a certain ODE, as shown in the next proposition.

\begin{proposition}\label{thm:eq_diff}
Fix $x \in M$. The function $J(t) = J_f(x,t) $ satisfies the linear ODE
\begin{equation*}
J''(t)+J'(t)+ \mu J(t) = e^{-t} \left[  Vf \circ \phi^X_{-t}(x) - Vf \circ \phi^X_{-t} \circ h_1(x) \right], 
\end{equation*}
with the intial conditions
$$
J(0) = \int_0^1 f \circ h_s(x) \diff s,\ \ \ \ \ J'(0) = \int_0^1 (-Xf) \circ h_s(x) \diff s.
$$
\end{proposition}
\begin{proof}
By assumption, $Xf$, $X^2f$, and $UVf$ are continuous, hence bounded, functions.
We can then write
\begin{equation}\label{eq:J_primo}
J'(t) = \int_0^1 \frac{\diff}{\diff t} f\circ \phi^X_{-t}\circ h_s(x) \diff s =  \int_0^1 (-Xf) \circ \phi^X_{-t}\circ h_s(x) \diff s,
\end{equation}
and
\begin{equation}\label{eq:J_secondo}
J''(t) =  \int_0^1 \frac{\diff^2}{(\diff t)^2}  f\circ \phi^X_{-t}\circ h_s(x) \diff s =  \int_0^1 \frac{\diff}{\diff t} (-Xf)\circ \phi^X_{-t}\circ h_s(x) \diff s = \int_0^1 X^2f \circ \phi^X_{-t}\circ h_s(x) \diff s.
\end{equation}
Let us also notice that, by Lemma \ref{lemma:c_o_v}, 
$$
\int_0^1 (UVf) \circ \phi^X_{-t}\circ h_s(x) \diff s = e^{-t}\left[ Vf\circ h_{e^t} \circ \phi^X_{-t}(x) - Vf \circ \phi^X_{-t}(x)\right].
$$
Therefore, we compute
\begin{equation*}
\begin{split}
\mu J(t) &= \int_0^1 \mu f \circ \phi^X_{-t}\circ h_s(x) \diff s = \int_0^1 (\square f) \circ \phi^X_{-t}\circ h_s(x) \diff s \\
&= \int_0^1 (-X^2f+Xf -UVf) \circ \phi^X_{-t}\circ h_s(x) \diff s \\
&= - J''(t) -J'(t) - e^{-t}\left[ Vf \circ \phi^X_{-t}\circ h_{1}(x) - Vf \circ \phi^X_{-t}(x)\right].
\end{split}
\end{equation*}
The initial conditions are clear from the definition of $J(t)$ and from \eqref{eq:J_primo} and \eqref{eq:J_secondo}.
\end{proof}

For the sake of notation, let us call
\begin{equation*}
G(t) = G_f(x,t) :=  Vf \circ \phi^X_{-t}(x) - Vf \circ \phi^X_{-t} \circ h_1(x).
\end{equation*}
Note that 
\begin{equation}\label{eq:estimate_G}
|G_f(x,t)| \leq 2 \|Vf\|_\infty \leq 2 \|f\|_{\mathscr{C}^2}.
\end{equation}

By Proposition \ref{thm:eq_diff}, in order to find an expression for $J_f(x,t)$, we need to solve the ODE
\begin{equation}\label{eq:eq_diff}
J''(t)+J'(t)+ \mu J(t) = e^{-t}G(t).
\end{equation}
Its solution, with initial conditions $J(0)$ and $J'(0)$, can be written explicitly and depends on the complex numbers 
$$
z_{\pm} = -\frac{1\pm \nu}{2}, \text{\ \ \ where\ \ \ } \nu = \sqrt{1-4\mu} \in \R_{\geq 0} \cup i\R_{>0},
$$
which are the roots of the characteristic polynomial $P_\mu(z):= z^2 + z +\mu$ of \eqref{eq:eq_diff}.

\section{Positive Casimir parameters}\label{section:positive_param}

In this section, we prove parts (i), (ii), and (iii) of Theorem \ref{thm:main1}, namely in the case that the Casimir eigenvalue $\mu$ is strictly positive. The case $\mu \leq 0$ will be treated in the next section.

\subsection{The principal series}\label{sec:princ}

When $\mu > 1/4$, then $\nu \in i \R_{>0}$ and $P_\mu(z)$ has two complex conjugate roots $z_{\pm} \in \mathbb{C}$ with real part equal to $-1/2$ and imaginary part $\Im \nu / 2 = \sqrt{4 \mu -1} / 2 >0$. 
The solution of \eqref{eq:eq_diff} is
\begin{equation}\label{eq:sol_1_princ}
\begin{split}
J(t) = &e^{-\frac{t}{2}} \cos\left(\frac{\Im\nu}{2}t \right) \left( - \frac{2}{\Im\nu} \int_0^t e^{-\frac{\xi}{2}} \sin\left(\frac{\Im \nu}{2}\xi \right)G(\xi) \diff \xi + J(0) \right) \\
&+e^{-\frac{t}{2}} \sin\left(\frac{\Im\nu}{2}t \right) \left( \frac{2}{\Im\nu} \int_0^t e^{-\frac{\xi}{2}} \cos\left(\frac{\Im\nu}{2}\xi \right)G(\xi) \diff \xi  + \frac{1}{\Im\nu} J(0) + \frac{2}{\Im\nu}J'(0) \right).
\end{split}
\end{equation}
Since the integrals in the expressions above are absolutely convergent for $t \to \infty$, we can define 
\begin{equation*}
\begin{split}
&\Dp{f}(x) = - \frac{2}{\Im\nu} \int_0^\infty e^{-\frac{\xi}{2}} \sin\left(\frac{\Im\nu}{2}\xi \right)G(\xi) \diff \xi + \left( \int_0^1 f \circ h_s(x) \diff s \right),\\
&\Dm{f}(x) = \frac{2}{\Im\nu} \int_0^\infty e^{-\frac{\xi}{2}} \cos\left(\frac{\Im\nu}{2}\xi \right)G(\xi) \diff \xi  + \frac{1}{\Im\nu} \left( \int_0^1 f \circ h_s(x) \diff s \right) - \frac{2}{\Im\nu}\left( \int_0^1 Xf \circ h_s(x) \diff s \right).
\end{split}
\end{equation*} 
In this way, we can rewrite \eqref{eq:sol_1_princ} as
\begin{equation}\label{eq:I_mu_princ}
J_{f}(x,t) = e^{-\frac{t}{2}} \cos\left(\frac{\Im\nu}{2}t \right) \Dp{f}(x) + e^{-\frac{t}{2}} \sin\left(\frac{\Im\nu}{2}t \right) \Dm{f}(x) + \mathcal{R}_{\mu}f(x,t),
\end{equation}
where 
\begin{equation*}
\begin{split}
\mathcal{R}_{\mu}f(x,t) = \frac{2}{\Im\nu} e^{-\frac{t}{2}} \Big[ &\cos\left(\frac{\Im\nu}{2}t \right) \int_t^\infty e^{-\frac{\xi}{2}} \sin\left(\frac{\Im\nu}{2}\xi \right)G(\xi) \diff \xi \\
&- \sin\left(\frac{\Im\nu}{2}t \right) \int_t^\infty e^{-\frac{\xi}{2}} \cos\left(\frac{\Im\nu}{2}\xi \right)G(\xi) \diff \xi \Big].
\end{split}
\end{equation*} 
Using \eqref{eq:estimate_G}, one can easily check that 
\begin{equation*}
\begin{split}
&\|\Dp{f}\|_{\infty} \leq \left(\frac{8}{\Im\nu} +1 \right) \|f\|_{\mathscr{C}^2}, \quad \|\Dm{f}\|_{\infty} \leq \frac{11}{\Im\nu} \|f\|_{\mathscr{C}^2},\\
& \quad \text{ and } \quad |\mathcal{R}_\mu f(x,t) | \leq \frac{16}{\Im\nu}  e^{-t}\|f\|_{\mathscr{C}^2}.
\end{split}
\end{equation*}
To deduce the expression of Theorem \ref{thm:main1}-(i) for the ergodic average of $f$, one simply needs to use~\eqref{eq:aver_equals_I}.

Applying the trivial estimates $|\cos(\frac{\Im\nu}{2} t)| \leq 1$ and $|\sin( \frac{\Im\nu}{2} t)| \leq  \frac{\Im\nu}{2} t$, from \eqref{eq:sol_1_princ} we also deduce
\begin{equation*}
\begin{split}
|J_f(x,t)| &\leq  e^{-\frac{t}{2}} \left( 2 \int_0^t e^{-\frac{\xi}{2}} \xi \diff \xi +1\right) \|f\|_{\mathscr{C}^2} +  t e^{-\frac{t}{2}} \left( 2 \int_0^t e^{-\frac{\xi}{2}} \diff \xi +\frac{3}{2}\right) \|f\|_{\mathscr{C}^2}\\
& \leq 9 e^{-\frac{t}{2}} \|f\|_{\mathscr{C}^2} + 6t e^{-\frac{t}{2}} \|f\|_{\mathscr{C}^2},
\end{split}
\end{equation*} 
which proves \eqref{eq:aux_thm_1_princ}, again recalling that $t = \log T$.

\subsection{The case $\mu = 1/4$}\label{sec:14}

If $\mu = 1/4$, then $-1/2$ is a double root of $P_\mu(z)$. In this case, the solution of \eqref{eq:eq_diff} is 
\begin{equation*}
\begin{split}
J(t) = &e^{-\frac{t}{2}} \left( -  \int_0^t \xi e^{-\frac{\xi}{2}} G(\xi) \diff \xi  +  J(0) \right) + t e^{-\frac{t}{2}} \left( \int_0^t e^{-\frac{\xi}{2}} G(\xi) \diff \xi + \frac{1}{2} J(0) + J'(0) \right).
\end{split}
\end{equation*}
As before, we can define
\begin{equation*}
\begin{split}
&\mathcal{D}_{1/4}^{+}f(x) = - \int_0^\infty \xi e^{-\frac{\xi}{2}} G(\xi) \diff \xi + \left( \int_0^1 f \circ h_s(x) \diff s \right),\\
&\mathcal{D}_{1/4}^{-}f(x) = \int_0^\infty e^{-\frac{\xi}{2}} G(\xi) \diff \xi + \frac{1}{2} \left( \int_0^1 f \circ h_s(x) \diff s \right) - \left( \int_0^1 Xf \circ h_s(x) \diff s \right),
\end{split}
\end{equation*}
so that we obtain
\begin{equation}\label{eq:I_mu_14}
J_{f}(x,t) = e^{-\frac{t}{2}} \mathcal{D}_{1/4}^{+}f(x) + t e^{-\frac{t}{2}} \mathcal{D}_{1/4}^{-}f(x) + \mathcal{R}_{1/4}f(x,t),
\end{equation}
where 
$$
\mathcal{R}_{1/4}f(x,t) = - t e^{-\frac{t}{2}} \int_t^{\infty} e^{-\frac{\xi}{2}} G(\xi) \diff \xi  +  e^{-\frac{t}{2}} \int_t^{\infty} \xi e^{-\frac{\xi}{2}} G(\xi) \diff \xi.
$$
By \eqref{eq:aver_equals_I}, we deduce the expression of Theorem \ref{thm:main1}-(ii) for $\aver{f}$.

One can easily check that 
\begin{equation*}
\begin{split}
&\|\mathcal{D}_{1/4}^{+}f\|_{\infty} \leq 9\|f\|_{\mathscr{C}^2}, \quad \|\mathcal{D}_{1/4}^{-}f\|_{\infty} \leq 6 \|f\|_{\mathscr{C}^2}, \\
&\text{ and } |\mathcal{R}_{1/4}f(x,t) | \leq 8 \|f\|_{\mathscr{C}^2} (t+2)e^{-t}.
\end{split}
\end{equation*}

\subsection{The complementary series}\label{sec:comp}

Finally, in the case $0< \mu < 1/4$, the characteristic polynomial $P_\mu(z)$ has two distinct real roots $z_{\pm} \in (-1,0)$. The solution of \eqref{eq:eq_diff} is
\begin{equation}\label{eq:sol_1_comp}
\begin{split}
J(t) = &e^{-\frac{1+\nu}{2}t} \left( - \frac{1}{\nu} \int_0^t e^{-\frac{1-\nu}{2}\xi} G(\xi) \diff \xi - \frac{1-\nu}{2\nu} J(0) - \frac{1}{\nu}J'(0) \right) \\
&+ e^{-\frac{1-\nu}{2}t} \left( \frac{1}{\nu} \int_0^t e^{-\frac{1+\nu}{2}\xi} G(\xi) \diff \xi + \frac{1+\nu}{2\nu} J(0) + \frac{1}{\nu}J'(0) \right).
\end{split}
\end{equation}
Once again, we define 
$$
\Dpm{f}(x) = \mp \frac{1}{\nu} \int_0^\infty e^{-\frac{1 \mp \nu}{2}\xi} G(\xi) \diff \xi \mp \frac{1 \mp \nu}{2\nu} \left( \int_0^1 f \circ h_s(x) \diff s \right) \pm \frac{1}{\nu} \left( \int_0^1 Xf \circ h_s(x) \diff s \right).
$$
Thus, we rewrite \eqref{eq:sol_1_comp} as 
\begin{equation}\label{eq:I_mu_comp}
J_f(x,t) = e^{-\frac{1+\nu}{2}t} \Dp{f}(x) + e^{-\frac{1-\nu}{2}t} \Dm{f}(x) + \mathcal{R}_\mu f(x,t),
\end{equation}
where 
$$
\mathcal{R}_\mu f(x,t) = \frac{1}{\nu} e^{-\frac{1+\nu}{2}t} \int_t^{\infty} e^{-\frac{1-\nu}{2}\xi} G(\xi) \diff \xi - \frac{1}{\nu} e^{-\frac{1-\nu}{2}t} \int_t^{\infty} e^{-\frac{1+\nu}{2}\xi} G(\xi) \diff \xi.
$$
It is immediate to check that 
\begin{equation*}
\begin{split}
&\|\Dp{f}\|_{\infty} \leq \frac{6}{\nu(1-\nu)} \|f\|_{\mathscr{C}^2}, \quad \|\Dm{f}\|_{\infty} \leq \frac{6}{\nu} \|f\|_{\mathscr{C}^2},\\
&\text{ and } |\mathcal{R}_\mu f(x,t) | \leq \frac{8}{\nu(1-\nu^2)}  \|f\|_{\mathscr{C}^2} e^{-t}.
\end{split}
\end{equation*}

We now show that
\begin{equation}\label{eq:claim_for_comp}
|J_f(x,t)| \leq \frac{6}{1-\nu} t e^{-\frac{1-\nu}{2}t}  \|f\|_{\mathscr{C}^2} + \frac{9}{(1-\nu)^2} e^{-\frac{1-\nu}{2}t}  \|f\|_{\mathscr{C}^2}, 
\end{equation}
which implies \eqref{eq:aux_thm_1_comp}.
We rewrite \eqref{eq:sol_1_comp} as follows
\begin{equation*}
\begin{split}
J(t) = \frac{e^{-\frac{1-\nu}{2}t}}{\nu} \Bigg[& - e^{-\nu t} \left( \int_0^t e^{-\frac{1-\nu}{2}\xi} G(\xi) \diff \xi + \frac{1-\nu}{2} J(0)+J'(0) \right) \\
&+ \int_0^t e^{-\frac{1+\nu}{2}\xi} G(\xi) \diff \xi + \frac{1+\nu}{2} J(0) + J'(0)  \Bigg].
\end{split}
\end{equation*}
By adding and subtracting $\int_0^t e^{-\frac{1-\nu}{2}\xi} G(\xi) \diff \xi$ inside the brackets, we get
\begin{equation*}
\begin{split}
|J(t)| \leq \frac{e^{-\frac{1-\nu}{2}t}}{\nu} \Bigg[& |1 - e^{-\nu t}| \left( \int_0^t e^{-\frac{1-\nu}{2}\xi} |G(\xi)| \diff \xi + \frac{1}{2} |J(0)| +|J'(0)| \right) \\
&+ \int_0^t | e^{-\frac{1+\nu}{2}\xi} - e^{-\frac{1-\nu}{2}\xi}| \, |G(\xi)| \diff \xi + \nu |J(0)|  \Bigg],
\end{split}
\end{equation*}
from which, since $|1-e^{-\nu t}| \leq \nu t$, the claim \eqref{eq:claim_for_comp} follows.

\subsection{The H\"{o}lder regularity of $\Dpm{f}$}

In order to complete the proof of parts (i), (ii), and (iii) of Theorem \ref{thm:main1}, the only thing left to prove is the claim on the regularity of the functions  $\Dpm{f}$.
Let us start with the case $\mu \neq 1/4$.
Fix $x \in M$ and let $y$ be a point at distance $r \in (0,1)$ from $x$. We can write $y = x \exp(rW)$, where $W = a_VV+a_XX+a_UU$ with $|a_V|, |a_X|, |a_U| \leq 1$.
We now consider the expressions for $\Dpm{f}$ we found in \S\ref{sec:princ} and \S\ref{sec:comp}:
it is clear that the functions 
$$
\int_0^1 f \circ h_s(x) \diff s \text{\ \ \ and\ \ \ }\int_0^1 Xf \circ h_s(x) \diff s
$$
are (at least) of class $\mathscr{C}^1$, thus we deduce that there exists a constant $C_f$ depending only on $f$ such that
$$
|\Dpm{f}(y) - \Dpm{f}(x)| \leq C_{f} \left( r \|f\|_{\mathscr{C}^1} + \int_0^{\infty} e^{-a \xi} |G_{f}(y,\xi) -G_{f}(x,\xi) | \diff \xi \right),
$$
where $a =  \frac{1\mp \Re \nu}{2}$. 
%
In order to prove the claim on the H\"{o}lder regularity of $\Dpm{f}$, it suffices to bound the integral in brackets above by $O(r^a)$.

\begin{lemma}
With the notation above,
$$
|G_{f}(y,\xi) -G_{f}(x,\xi) | \leq 6 \|f\|_{\mathscr{C}^2} \min\{1, r e^\xi\}.
$$
\end{lemma}
\begin{proof}
From the definition of $G_{f}(x,\xi)$, it follows that 
$$
|G_{f}(y,\xi) -G_{f}(x,\xi) |\leq r ( \| [D\phi^X_{-\xi}(W) ] Vf \|_\infty + \|[D\phi^X_{-\xi} \circ Dh_1 (W)] Vf \|_\infty).
$$ 
Since we can write $Dh_1 (W) = {\widetilde a_V}V+{\widetilde a_X}X+{\widetilde a_U}U$, for some $|{\widetilde a_V}|, |{\widetilde a_X}|, |{\widetilde a_U}| \leq 2$, we get the estimate
$$
|G_{f}(y,\xi) -G_{f}(x,\xi) |\leq 2 r \|(e^{-\xi}V + X + e^{\xi}U)Vf\|_\infty \leq 6 \|f\|_{\mathscr{C}^2} r e^\xi.
$$
This and \eqref{eq:estimate_G} concludes the proof.
\end{proof}
The conclusion follows from the following elementary lemma, applied to $F=|G_{f}(y, \cdot) - G_{f}(x, \cdot)|$.
\begin{lemma}
Let $F(\xi)$ be a continuous and positive function satisfying $F(\xi) \leq C_0 \min\{ 1, e^\xi r\}$ for some $C_0>0$ and $r \in (0,1)$. Then, for all $a\in (0,1)$ we have
$$
\int_0^\infty e^{-a \xi} \, F(\xi) \diff \xi \leq C_0 \max\{(1-a)^{-1}, a^{-1}\} r^a.
$$ 
\end{lemma}
\begin{proof}
Let $A = -\log r >0$. We have
\begin{equation*}
\begin{split}
\int_0^\infty e^{-a \xi} \, F(\xi) \diff \xi &= \int_0^A  e^{-a \xi} \, F(\xi) \diff \xi + \int_A^{\infty} e^{-a \xi} \, F(\xi) \diff \xi \\
& \leq C_0 r \int_0^A  e^{(1-a) \xi} \diff \xi + C_0 \int_A^{\infty} e^{-a \xi} \diff \xi \\
&\leq \frac{C_0}{1-a}e^{A(1-a)} r + \frac{C_0}{a} e^{-Aa} \leq C_0 \max\{(1-a)^{-1}, a^{-1}\} r^a.
\end{split}
\end{equation*}
\end{proof}

In the case $\mu=1/4$, the only difference is for $\mathcal{D}^{+}_{1/4}f$, in which case one gets the bound
$$
|\mathcal{D}^{+}_{1/4}f(y) - \mathcal{D}^{+}_{1/4}f(x)| \leq C_{f} \left( r \|f\|_{\mathscr{C}^1} + \int_0^{\infty} \xi e^{- \xi/2} |G_{f}(y,\xi) -G_{f}(x,\xi) | \diff \xi \right). 
$$
The integral in the right-hand side above can be estimated by $O(-\sqrt{r} \log r)$ in the same way as we did before by using the following easy lemma.
\begin{lemma}
Let $F(\xi)$ be a continuous and positive function satisfying $F(\xi) \leq C_0 \min\{ 1, e^\xi r\}$ for some $C_0>0$ and $r \in (0,1)$. Then, we have
$$
\int_0^\infty \xi e^{- \xi/2} \, F(\xi) \diff \xi \leq -8C_0 \sqrt{r} \log r.
$$
\end{lemma}

\section{The Discrete Series}\label{section:discrete_ser}

In this section, we consider the case $\mu \leq 0$ and we will prove Theorem \ref{thm:main1}-(iv) and (v). 
It is well-known that the only possible non-positive eigenvalues of the Casimir operator are given by 
$$
\mu = -n^2+n, \quad \text{ for }n \in \Z,\ n \geq 1.
$$

\subsection{Case $n=1$}
We consider the case that $\mu = 0$. The solution of \eqref{eq:eq_diff} is
\begin{equation}\label{eq:sol_1_zero}
J(t) = \left( J(0) + J'(0) + \int_0^\infty e^{-\xi}G(\xi) \diff \xi \right) - \int_t^\infty e^{-\xi}G(\xi) \diff \xi - e^{-t} J'(0) - e^{-t}\int_0^t G(\xi) \diff \xi.
\end{equation}
Let us assume that $\vol(f)=0$. Unique ergodicity of the horocycle flow and equality \eqref{eq:aver_equals_I} imply that $\|J(\log T) \|_\infty = \| A_f(\cdot, T) \|_\infty  \to 0$; thus the constant term in brackets in \eqref{eq:sol_1_zero} is zero.
Recalling \eqref{eq:estimate_G}, we deduce that 
$$
\left\lvert J_f(x,t) + e^{-t}\int_0^t G_f(x,\xi) \diff \xi \right\rvert \leq 3 e^{-t} \|f\|_{\mathscr{C}^2}.
$$
By the definition of $G(\xi) = G_f(x,\xi)$ and \eqref{eq:aver_equals_I}, we conclude 
$$
\left\lvert \aver{f} - \frac{1}{T} \int_0^{\log T} \left( Vf \circ \phi^X_\xi \circ h_T(x) - Vf\circ \phi^X_\xi (x) \right) \diff \xi \right\rvert \leq \frac{3}{T} \|f\|_{\mathscr{C}^2},
$$
which proves Theorem \ref{thm:main1}-(iv).

\begin{remark}\label{remark:Vf_cob}
It is easy to see that Theorem \ref{thm:main1}-(iv) implies that if $f$ is not a measurable coboundary for $U$, then $Vf$ is not a measurable coboundary for $X$. Indeed, if this was not the case, by Livsic's Theorem, $Vf$ is a continuous coboundary for $X$. By Theorem \ref{thm:main1}-(iv), the horocycle ergodic integrals of $f$ are \emph{uniformly bounded}. Since $h_t$ is minimal, by the Gottschalk-Hedlund Theorem, we deduce that $f$ is a continuous coboundary for $U$, which is a contradiction with the assumption.
\end{remark}

\subsection{Case $n\geq 2$}

As we recalled, if $\mu <0$, then $\mu \leq -2$ and $\nu \geq 3$. 
In this case, the general solution of \eqref{eq:eq_diff} is the same as in \eqref{eq:sol_1_comp}.
Rearranging the terms, it is easy to see that there is a constant $C_f= C(n,f,x)$ (depending on $n$, $f$ and $x$) such that
\begin{equation}\label{eq:step1_discrete}
\left\lvert J(t) - e^{-\frac{1-\nu}{2}t} \left(  \frac{1}{\nu} \int_0^{\infty} e^{-\frac{1+\nu}{2}\xi} G(\xi) \diff \xi + \frac{1+\nu}{2\nu} J(0) + \frac{1}{\nu}J'(0)\right)\right \rvert \leq C_f e^{-t}.
\end{equation}
Note that the second term in the left-hand side in \eqref{eq:step1_discrete} diverges for $t \to \infty$, unless the constant in brackets is zero. 
Indeed, this must be the case, since we have the a priori estimate
$$
|J(t)| \leq \| A_f(x,e^t) \|_\infty \leq \|f\|_\infty.
$$

Therefore, again from \eqref{eq:sol_1_comp}, we conclude that 
\begin{equation}\label{eq:estimate_discrete}
\begin{split}
|J_{f}(x,t)| \leq & \left\lvert e^{-\frac{1+\nu}{2}t} \left( - \frac{1}{\nu} \int_0^t e^{-\frac{1-\nu}{2}\xi} G(\xi) \diff \xi - \frac{1-\nu}{2\nu} J(0) - \frac{1}{\nu}J'(0) \right) \right \rvert \\
&+ \left\lvert \frac{1}{\nu}  e^{-\frac{1-\nu}{2}t} \int_t^{\infty} e^{-\frac{1+\nu}{2}\xi} G(\xi) \diff \xi\right\rvert \\
\leq & 5e^{-t}\|f\|_{\mathscr{C}^2}.
\end{split}
\end{equation}

We showed that the ergodic integrals of $f$ are \emph{uniformly bounded}, namely
$$
\left\lvert \int_0^T f \circ h_t (x) \diff t \right\rvert = |T \aver{f}| = |T J_{f}(\phi^X_{\log T}(x), \log T)| \leq 5\|f\|_{\mathscr{C}^2}.
$$
By the Gottschalk-Hedlund Theorem, this implies that $f$ is a continuous coboundary.

\section{Flaminio and Forni's Theorem and the invariant distributions}\label{section:FF}

We now prove Theorem \ref{thm:main_FF}. As we already pointed out, the proof follows from Theorem \ref{thm:main1} and some basic facts from the harmonic analysis of the Lie group $\PSL(2,\R)$.

\subsection{Preliminaries}

Let $\mathcal{U}(H)$ the group of unitary transformations of the Hilbert space $H= L^2(M)$. 
We denote by $\rho \colon \PSL(2,\R) \to \mathcal{U}(H)$ the \emph{right regular representation} defined by
$$
[\rho(g)f](x) = f(xg),
$$
for all $g \in \PSL(2,\R)$, $f \in H$, and $x \in M$.

It is a standard fact that $H$ splits into a direct sum of countably many irreducible subspaces $H_{i}$ for $i \in \mathcal{I}$, and on each of these subspaces the Casimir operator acts as a multiple of the identity, namely for all $i \in \mathcal{I}$ there exists $\mu \in \Spec(\square)$ such that 
$$
\square f = \mu f \quad \text{ for all $f \in H_i$ of class $\mathscr{C}^2$}.
$$
More precisely, the Casimir eigenvalue $\mu$ has the form
$$
\mu = \frac{1-\nu^2}{4},
$$
where $\nu \in i\R$ (a \emph{principal series} representation), or $\nu \in (-1,1) \setminus \{0\}$ (a \emph{complementary series} representation), or $\nu = 2n-1$, $n \in 2\Z_{>0}$ (a \emph{discrete series} representation).
The principal and complementary series representations with parameters $\nu$ and $-\nu$ are isomorphic, hence we can restrict ourselves to the case $\nu \in i\R_{\geq 0} \cup (0,1) \cup \{ 2n-1 : n \in \Z_{>0}\}$, and write
\begin{equation}\label{eq:ort_sum_hmu}
H= \bigoplus_{\mu \in \Spec(\square)} H_{\mu},
\end{equation}
where  $H_\mu$ is the orthogonal sum of all the irreducible representations $H_i$ of the same parameter $\mu$, so that each $\mu$ appears only once in the decomposition \eqref{eq:ort_sum_hmu}. 

We will need the following fact.
\begin{lemma}\label{lemma:Weyl}
The infinite sum
\[
\sum_{\mu \in \Spec(\square)} \frac{1}{(1+|\mu|)^2}
\]
converges.
\end{lemma}
\begin{proof}
As we mentioned above, the eigenvalues $\mu \leq 0$ can be written as $\mu = -n^2+n$ for $n \in \Z_{>0}$, hence the sum of $(1+|\mu|)^{-2}$ for all $\mu\leq 0$ is finite. 
On the other hand, the positive part of the spectrum $0<\mu_1 <\mu_2< \dots$ of $\square$ coincides with the spectrum of the Laplace-Beltrami operator on the hyperbolic surface $S =\Gamma \backslash \Hbb$. By Weyl's Law, the number of these eigenvalues $\mu_n \in \Spec(\square) \cap \R_{>0}$, counted with multiplicity, in any given interval of the form $(0,R)$ grows asymptotically linearly in $R$. This implies that $(1+\mu_n)^{-1} = O(n^{-1})$, in particular the sequence is square summable. 
\end{proof}
By Lemma \ref{lemma:Weyl}, we can then define the constant $C_{\Spec}>0$ by
\[
C^2_{\Spec}:=\sum_{\mu \in \Spec(\square)} \frac{1}{(1+|\mu|)^2}.
\]

At the level of Sobolev spaces, from \eqref{eq:ort_sum_hmu} we have an induced decomposition 
$$
W^r(M) = W^r(H) = \bigoplus_{\mu \in \Spec(\square)} W^r(H_{\mu}) \quad \text{for all $r \geq 0$}.
$$
Hence, for all $r \geq 4$ and $f \in W^r(M)$, we can write
\begin{equation}\label{eq:decomp_f_irreps}
f = \sum_{\mu \in \Spec(\square)} f_\mu, \quad \text{ with } \quad f_\mu \in  W^r(H_{\mu}) \subset \mathscr{C}^2(M), \quad \text{ and } \quad \|f\|^2_{W^r} = \sum_{\mu \in \Spec(\square)} \|f_\mu\|^2_{W^r}.
\end{equation}
\begin{lemma}\label{lemma:w4_to_w6}
Let $f \in W^r(M)$ be as above. For any $\mu \in \Spec(\square)$, we have 
\[
(1+|\mu|) \|f_\mu\|^2_{W^{r-1}} \leq \|f_\mu\|^2_{W^r}.
\]
\end{lemma}
\begin{proof}
Let us remark that it is enough to prove the lemma under the assumption that $H_\mu$ is an irreducible subspace. 
For any given $f_\mu \in H_\mu$, we can write $f_\mu = \sum_{k \in I_\mu} u_{\mu, k}$, where $u_{\mu, k}$ are mutually orthogonal eigenvectors of $\Theta$, namely they satisfy $\Theta u_{\mu, k} = iku_{\mu, k}$, and $I_\mu \subseteq \Z$ if $\mu >0$ and $I_\mu \subseteq n+ \Z_{\geq 0}$ if $\mu = -n^2+n \leq 0$ (see, e.g., \cite[\S2]{FlaFo}).

By the definition \eqref{eq:def_inn_prod_Wr} of the inner product in $W^r$, we have 
\[
\|f_\mu\|^2_{W^r} = \langle (\Id+\Delta) f_\mu, f_\mu \rangle_{W^{r-1}} = \|f_\mu\|^2_{W^{r-1}} + \mu \|f_\mu\|^2_{W^{r-1}} +  \langle -2\Theta^2 f_\mu, f_\mu \rangle_{W^{r-1}},
\]
where we used the fact that we can write the Laplacian as $\Delta = \square - 2 \Theta^2$. 
If $\mu >0$
\[
\langle -2\Theta^2 f_\mu, f_\mu \rangle_{W^{r-1}} = \sum_{k \in I_\mu} 2k^2 \|u_{\mu, k}\|_{W^{r-1}} \geq 0,
\]
otherwise, if $\mu = -n^2+n \leq 0$, 
\[
\langle -2\Theta^2 f_\mu, f_\mu \rangle_{W^{r-1}} = \sum_{k \in I_\mu} 2k^2 \|u_{\mu, k}\|_{W^{r-1}} \geq 2n^2 \|f_\mu\|^2_{W^{r-1}}\geq -2\mu \|f_\mu\|^2_{W^{r-1}}.
\]
In both cases, we conclude that $\|f_\mu\|^2_{W^r} \geq (1+|\mu|)\|f_\mu\|^2_{W^{r-1}}$.
\end{proof}

Let $f \in W^6(M)$, and let us decompose $f$ as in \eqref{eq:decomp_f_irreps}. We now show that we can bound $\sum_{\mu \in \Spec(\square)} \|f_\mu\|_{\mathscr{C}^2}$ in terms of the $W^6$ norm of $f$. 

\begin{lemma}\label{lemma:bound_no_square}
Let $f \in W^6(M)$ and write $f= \sum_{\mu \in \Spec(\square)} f_\mu$ as in \eqref{eq:decomp_f_irreps}. Then,
\[
\sum_{\mu \in \Spec(\square)} \|f_\mu\|_{\mathscr{C}^2} \leq C_{\emb} \, C_{\Spec}  \|f\|_{W^6}.
\]
\end{lemma}
\begin{proof}
From the Cauchy-Schwarz inequality and using Lemma \ref{lemma:Weyl}, we deduce
\begin{equation*}
\sum_{\mu \in \Spec(\square)} \|f_\mu\|_{\mathscr{C}^2} \leq C_{\emb} \sum_{\mu \in \Spec(\square)} \|f_\mu\|_{W^4} \leq C_{\emb} \, C_{\Spec} \left(\sum_{\mu \in \Spec(\square)} (1+|\mu|)^2 \|f_\mu\|^2_{W^4}\right)^{1/2}. 
\end{equation*}
Applying now Lemma \ref{lemma:w4_to_w6}, we get
\[
\sum_{\mu \in \Spec(\square)} \|f_\mu\|_{\mathscr{C}^2} \leq C_{\emb} \, C_{\Spec} \left(\sum_{\mu \in \Spec(\square)} \|f_\mu\|^2_{W^6}\right)^{1/2} = C_{\emb} \, C_{\Spec}  \|f\|_{W^6},
\]
which completes the proof.
\end{proof}
We are ready to give the proof of our version of Flaminio and Forni's Theorem.

\subsection{Proof of Theorem \ref{thm:main_FF}}

Let us define the constant $C_M>0$ by 
\begin{equation}\label{eq:defin_CM}
C_M = 20 \, {\widetilde C_M} C_{\emb} C_{\Spec}, \text{\ \ \ where\ \ \ } {\widetilde C_M}= \max \left\{ |\nu|^{-1}, |1-\nu|^{-1} : \mu \in \sigma_{\comp} \cup \sigma_{\princ} \right\} < \infty.
\end{equation}

Let $f \in W^6(M)$ and consider the decomposition as in \eqref{eq:decomp_f_irreps}. 
Since, by Lemma \ref{lemma:bound_no_square}, for all $x \in M$
$$
\sum_{\mu \in \Spec(\square)} |f_\mu(x)| \leq \sum_{\mu \in \Spec(\square)} \|f_\mu\|_{\mathscr{C}^2} \leq C_{\emb} \, C_{\Spec}  \|f\|_{W^6},
$$
we have
\begin{equation}\label{eq:int_A}
\begin{split}
&\frac{1}{T} \int_0^T f \circ h_s(x) \diff s = \int_M f \diff \vol +  \sum_{\mu \in \Spec(\square) } A_{\mu}(x,T), \\
&\text{where\ \ \ } A_{\mu}(x,T) = \frac{1}{T} \int_0^T f_\mu \circ h_s(x) \diff s, \text{\ \ \ and\ \ \ } \vol(f_\mu)=0.
\end{split}
\end{equation}

We can now apply the results of Theorem \ref{thm:main1} to each of the components $f_\mu$.
In particular, the expression for the ergodic average of $f$ in Theorem \ref{thm:main_FF} is satisfied by defining
$$
\mathcal{R}f(x,T) = \sum_{\mu \in \Spec(\square) \cap \R_{>0}} \mathcal{R}_\mu (f_\mu) (x, T) + \sum_{\mu \in \Spec(\square) \cap \R_{\leq 0}} A_{f_\mu}(x,T).
$$
By Theorem \ref{thm:main1} and by Lemma \ref{lemma:bound_no_square}, we have
\begin{equation*}
\begin{split}
|\mathcal{R}f(x,T)| &\leq \sum_{\mu \in \Spec(\square) \cap \R_{>0}}| \mathcal{R}_\mu (f_\mu)(x, T)| + \sum_{\mu \in \Spec(\square) \cap \R_{\leq 0}} |A_{f_\mu}(x,T)| \\
&\leq 16 \, {\widetilde C_M} \frac{1+ \log T}{T} \sum_{\mu \in \Spec(\square) \cap \R_{>0}} \|f_\mu\|_{\mathscr{C}^2}+ 5  \frac{1+\log T}{T} \sum_{\mu \in \Spec(\square) \cap \R_{\leq 0}} \|f_\mu\|_{\mathscr{C}^2} \\
&\leq C_M\|f \|_{W^6} \frac{1+ \log T}{T},
\end{split}
\end{equation*}
and also
$$
\sum_{\mu \in \Spec(\square) \cap \R_{>0}} \| \Dpm{f}\|_{\infty} \leq 11 {\widetilde C_M} \sum_{\mu \in \Spec(\square) \cap \R_{>0}}\|f_\mu \|_{\mathscr{C}^2}  \leq C_M \|f \|_{W^6}.
$$
The proof of Theorem \ref{thm:main_FF} is complete.

\begin{remark}
It is not difficult to see that it is possible to obtain a version of Theorem \ref{thm:main_FF} in which the constant $C_M$ does not depend on the maximum of $|\nu|^{-1}$ for $\mu \in \sigma_{\comp} \cup \sigma_{\princ}$, but only on the spectral gap (i.e., on the maximum of $|1-\nu|^{-1}$), at the price of an extra factor $\log T$: to this end, the bounds \eqref{eq:aux_thm_1_princ} and \eqref{eq:aux_thm_1_comp} in Theorem \ref{thm:main1}-(i), (iii) are needed to estimate the components of $f$ on irreducible subspaces of eigenvalues close to $1/4$. 
\end{remark}

\subsection{The action of the geodesic flow}\label{section:geo_action}

We now prove Proposition \ref{thm:eigenvectors} by computing $\Dpmp{Xf}$ for all fixed $\mu \in \Spec(\square) \cap \R_{>0}$.
In order to do this, we need to replace $f=f_\mu$ with $Xf= Xf_\mu$ in the expressions we found in \S\ref{section:positive_param}.

Note that, by \eqref{eq:J_primo} and \eqref{eq:J_secondo}, 
\begin{equation}\label{eq:JX}
\begin{split}
&J_{Xf}(0) = \int_0^1 Xf \circ h_s(x) \diff s = -J_f'(0), \text{\ \ \ and\ \ \ } \\
&J_{Xf}'(0)= \int_0^1 (-X^2f) \circ h_s(x) \diff s  =- J_f''(0).
\end{split}
\end{equation}

We start with a simple computation that will become useful later.
\begin{lemma}\label{lemma:aGx}
Let $\ell(\xi)$ be a smooth function such that $\ell, \ell' \in L^1(\R)$ and $\lim_{\xi \to \infty}\ell(\xi) =0$. Then,
$$
\int_0^\infty \ell(\xi) G_{Xf}(\xi) \diff \xi = \int_0^\infty [\ell(\xi)+\ell'(\xi)] G_f(\xi) \diff \xi + \ell(0)\left( \mu J_f(0) + J_f'(0) + J_f''(0) \right).
$$ 
\end{lemma}
\begin{proof}
By the commutation relation between $X$ and $V$, namely $[X,V]=-V$, we get
\begin{equation*}
\begin{split}
G_{Xf}(\xi) &= VXf \circ \phi^X_{-\xi}(x) - VXf \circ \phi^X_{-\xi} \circ h_1(x) \\
&= XVf \circ \phi^X_{-\xi}(x) +Vf \circ \phi^X_{-\xi}(x)  - XVf \circ \phi^X_{-\xi} \circ h_1(x) -Vf \circ \phi^X_{-\xi} \circ h_1(x) \\
&= G_f(\xi) - G_f'(\xi).
\end{split}
\end{equation*}
Integrating by parts, and using $\ell(\xi) \to 0$, we obtain
$$
\int_0^\infty \ell(\xi) G_{Xf}(\xi) \diff \xi = \int_0^\infty [\ell(\xi)+\ell'(\xi)] G_f(\xi) \diff \xi + \ell(0)G_f(0).
$$
We only need to rewrite $G_f(0)$. Evaluating \eqref{eq:eq_diff} at $0$, we have
\begin{equation*}
\begin{split}
G_f(0) &=\mu J_f(0) + J_f'(0) + J_f''(0),
\end{split}
\end{equation*}
which completes the proof.
\end{proof}

\begin{proof}[{Proof of Proposition \ref{thm:eigenvectors}}]
Let us assume $0<\mu< 1/4$. From the expressions for $\Dpm$ in \S\ref{sec:comp}, \eqref{eq:JX}, and Lemma \ref{lemma:aGx}, we get
\begin{equation*}
\begin{split}
\Dpmp{Xf} =&  \mp \frac{1}{\nu} \int_0^\infty e^{-\frac{1 \mp \nu}{2}\xi} G_{Xf}(\xi) \diff \xi \mp \frac{1 \mp \nu}{2\nu} J_{Xf}(0) \mp \frac{1}{\nu} J'_{Xf}(0)\\
=& \mp \frac{1}{\nu} \left( \int_0^\infty \big( 1 -\frac{1 \mp \nu}{2}\big) e^{-\frac{1 \mp \nu}{2}\xi} G_f(\xi) \diff \xi + \mu J_f(0) + J_f'(0) + J_f''(0) \right) \\
& \pm \frac{1 \mp \nu}{2\nu} J'_f(0) \pm \frac{1}{\nu} J''_f(0).
\end{split}
\end{equation*}
Using the fact that $(1+\nu)(1-\nu) = 4 \mu$, we conclude
$$
\Dpmp{Xf} = \frac{1 \pm \nu}{2} \left(\mp \frac{1}{\nu} \int_0^\infty e^{-\frac{1 \mp \nu}{2}\xi} G_f(\xi) \diff \xi \mp \frac{1 \mp \nu}{2\nu} J_f(0) \mp \frac{1}{\nu} J_f'(0) \right) = \frac{1 \pm \nu}{2} \Dpmp{f}.
$$

Let now $\mu = 1/4$. Using the expressions in \S\ref{sec:14}, by \eqref{eq:JX} and Lemma \ref{lemma:aGx}, we get
\begin{equation*}
\begin{split}
\mathcal{D}_{1/4}^{+}(Xf) &= - \int_0^\infty \xi e^{-\frac{\xi}{2}} G_{Xf}(\xi) \diff \xi + J_{Xf}(0)\\
&=- \frac{1}{2} \int_0^\infty \xi e^{-\frac{\xi}{2}} G_{f}(\xi) \diff \xi -\int_0^\infty e^{-\frac{\xi}{2}} G_{f}(\xi) \diff \xi -J_f'(0) \\
&=- \frac{1}{2} \int_0^\infty \xi e^{-\frac{\xi}{2}} G_{f}(\xi) \diff \xi -\mathcal{D}_{1/4}^{-}f +\frac{1}{2}J_f(0) + J_f'(0) -J_f'(0) = \frac{1}{2} \mathcal{D}_{1/4}^{+}f - \mathcal{D}_{1/4}^{-}f .
\end{split}
\end{equation*}
Similarly,
\begin{equation*}
\begin{split}
\mathcal{D}_{1/4}^{-}(Xf) &= \int_0^\infty e^{-\frac{\xi}{2}} G_{Xf}(\xi) \diff \xi + \frac{1}{2} J_{Xf}(0) + J_{Xf}'(0)\\
&=\frac{1}{2}\int_0^\infty e^{-\frac{\xi}{2}} G_{f}(\xi) \diff \xi + \frac{1}{4} J_f(0) + J'_f(0) + J_f''(0) - \frac{1}{2} J_{f}(0) - J_{f}''(0) = \frac{1}{2} \mathcal{D}_{1/4}^{+}f.
\end{split}
\end{equation*}

Finally, let $\mu > 1/4$. Lemma \ref{lemma:aGx} gives us
\begin{equation*}
\begin{split}
\int_0^\infty e^{-\frac{\xi}{2}} \sin\left(\frac{\Im\nu}{2}\xi \right) G_{Xf}(\xi) \diff \xi  = & \frac{1}{2} \int_0^\infty e^{-\frac{\xi}{2}} \sin\left(\frac{\Im\nu}{2}\xi \right) G_{f}(\xi) \diff \xi \\
&+ \frac{\Im\nu}{2}\int_0^\infty e^{-\frac{\xi}{2}} \cos\left(\frac{\Im\nu}{2}\xi \right) G_{f}(\xi) \diff \xi, \\ 
\end{split}
\end{equation*}
and
\begin{equation*}
\begin{split}
\int_0^\infty e^{-\frac{\xi}{2}} \cos\left(\frac{\Im\nu}{2}\xi \right) G_{Xf}(\xi) \diff \xi  = & \frac{1}{2} \int_0^\infty e^{-\frac{\xi}{2}} \cos\left(\frac{\Im\nu}{2}\xi \right) G_{f}(\xi) \diff \xi  \\
&- \frac{\Im \nu}{2}\int_0^\infty e^{-\frac{\xi}{2}} \cos\left(\frac{\Im\nu}{2}\xi \right) G_{f}(\xi) \diff \xi + \mu J_f(0) +  J_f'(0)+  J_f''(0).
\end{split}
\end{equation*}
Plugging these and \eqref{eq:JX} into the expressions in \S\ref{sec:princ} for $\Dpmp{Xf}$, we obtain
$$
\Dp{(Xf)}=  \frac{1}{2} \Dp{f} - \frac{\Im\nu}{2} \Dm{f}, \text{\ \ \ and\ \ \ }\Dm{(Xf)}=  \frac{1}{2} \Dm{f}  + \frac{\Im\nu}{2} \Dp{f},
$$
which concludes the proof.
\end{proof}

\subsection{The case of logarithmic growth}\label{sec:log_growth}

Let us consider the case of a function $f \in W^6(M)$ with zero integral such that $\Dpm{f} \equiv 0$ for all $\mu \in \Spec(\square) \cap \R_{>0}$.
From Theorem \ref{thm:main_FF}, it follows that the ergodic integral of $f$ up to time $t$ can be bounded by a constant times $\log t$. 
One can actually be more precise, since the terms in \eqref{eq:int_A} which are of order $\log t$ are only (possibly) two: the error term corresponding to the Casimir parameter $1/4$ and the main term corresponding to the Casimir parameter 0 (as in Theorem \ref{thm:main1}-(ii) and (iv)). With this in mind, we can easily prove the following lemma.

\begin{lemma}\label{lemma:log_growth}
Let $f \in W^6(M)$ be a real-valued function with $\vol(f)=0$. 
Assume that
\begin{itemize}
\item[(a)] either $1/4 \notin \Spec(\square)$ and $\Dpm{f} \equiv 0$ for all $\mu \in \Spec(\square) \cap \R_{>0}$,
\item[(b)] or all components $f_\mu$ of $f$ in \eqref{eq:decomp_f_irreps} corresponding to positive parameters $\mu >0$ are measurable coboundaries.
\end{itemize}
Then, for all $T \geq 1$ and for all $x \in M$, there exists $\mathcal{E}(x,T) \in \R$ with 
$$
| \mathcal{E}(x,T) | \leq C_M \|f\|_{W^4},
$$
such that 
$$
\int_0^T f \circ h_s(x) \diff s = \int_0^{\log T} \left( Vf_0 \circ \phi^X_{\xi}\circ h_T(x) - Vf_0\circ \phi^X_{\xi}(x) \right) \diff \xi + \mathcal{E}(x,T).
$$
Moreover, if $f$ is not a measurable coboundary for $U$, then $Vf_0$ is not a measurable coboundary for $X$.
\end{lemma}
\begin{proof}
In case (a), the proof of the formula is an immediate adaptation of the proof of Theorem \ref{thm:main_FF}. If (b) holds, then it follows from the work of Flaminio and Forni that the ergodic integrals of the components $f_\mu$ for positive $\mu$ are uniformly bounded, so that the conclusion follows from Theorem \ref{thm:main1}-(iv).

The proof of the last claim is the same as in Remark \ref{remark:Vf_cob}.
\end{proof}

\section{Spatial limit theorems}\label{sec:limit}

In this section, we prove the limit theorems for the horocycle integrals. The proof of Theorem \ref{thm:main_BF1} is an easy consequence of Theorem \ref{thm:main_FF}.

Let us recall some preliminary notions for the reader's convenience. Let $X,Y \colon \Omega \to \R$ be two random variables defined on the same probability space $(\Omega, \mathcal{B}, \mathbb{P})$ with associated probability measures $\nu_X$ and $\nu_Y$ respectively. 
The \emph{L\'{e}vy-Prokhorov distance} $\Levy$ between $\nu_X$ and $\nu_Y$ is defined by
$$
\Levy(\nu_X, \nu_Y) := \inf \{ \varepsilon > 0 : \nu_X(B) \leq \nu_Y(B_{\varepsilon}) + \varepsilon \text{\ and\ } \nu_Y(B) \leq \nu_X(B_{\varepsilon}) + \varepsilon \text{\ for all Borel sets\ } B \subset \R\},
$$
where $B_{\varepsilon}$ denotes the $\varepsilon$-neighbourhood of $B$.
The L\'{e}vy distance is a metrization of the topology of weak convergence of measures.
We will use the following simple fact.

\begin{lemma}\label{lemma:levy}
Let $T \colon \Omega \to \Omega$ be a probability preserving map. If $| X(\omega) - Y\circ T(\omega)| \leq \varepsilon$ for $\mathbb{P}$-almost every $\omega \in \Omega$, then $\Levy(\nu_X, \nu_Y) \leq \varepsilon$.
\end{lemma}

We now turn to our case $(\Omega, \mathbb{P}) = (M, \vol)$. 
Let $f \in W^6(M)$ be as in the assumption of the theorem, namely such that 
$$
\mu_f = \min \{ \mu \in \Spec(\square) \cap \R_{>0} : \Dm{f} \not\equiv 0 \}
$$
is finite. 
Let us first assume that $\mu_f <1/4$, and recall that we defined $\nu_f = \sqrt{1-4\mu_f} \in (0,1)$. Let $A$ be the random variable 
$$
A = T^{-\frac{1+\nu_f}{2}} \int_0^T f \circ h_t(x) \diff t,
$$
with $x \sim \vol$.
Let 
$$
\eta_0 = \min\{ \nu_f - \nu : \mu \in \Spec(\square) \cap (\mu_f, \infty) \} >0, \quad \text{ and } \quad \eta= \frac{1}{2} \min \left\{\eta_0, 1-\nu_f, \nu_f \right\}.
$$
By Theorem \ref{thm:main_FF}, 
\begin{equation*}
\begin{split}
\left\lvert A - \mathcal{D}^{-}_{\mu_f} f\circ \phi^X_{\log T} \right\rvert \leq & T^{- \nu_f} \| \mathcal{D}^{+}_{\mu_f} f\|_{\infty} + T^{- \frac{\eta_0}{2}} \left( \sum_{\mu \in \sigma_{\comp}} \| \Dpm{f} \|_{\infty} \right) \\
& + T^{- \frac{\nu_f}{2}}  \log T \left( \sum_{\mu \in \Spec(\square) \cap [1/4, \infty)} \| \Dpm{f} \|_{\infty} \right) +  T^{1-\frac{1-\nu_f}{2}}  \| \mathcal{R}f(\cdot, T) \|_{\infty} \\
\leq &T^{-\eta} (1+\log T) \left( \sum_{\mu \in \Spec(\square) \cap \R_{>0}} \| \Dpm{f} \|_{\infty} \right) + C_M \|f\|_{W^6}  T^{-\eta}(1+\log T)\\
\leq & 
2C_M \|f\|_{W^6}  T^{-\eta}(1+\log T).
\end{split}
\end{equation*}
Lemma \ref{lemma:levy} completes the proof for the case $\mu_f < 1/4$. 

If $\mu_f = 1/4$, let $A$ be the random variable
$$
A = (T^{\frac{1}{2}} \log T)^{-1} \int_0^T f \circ h_t(x) \diff t.
$$
Then, again by Theorem \ref{thm:main_FF}, 
\begin{equation*}
\begin{split}
\left\lvert A - \mathcal{D}^{-}_{\mu_f} f \circ \phi^X_{\log T} \right\rvert &\leq (\log T)^{-1} \left(  \| \mathcal{D}^{+}_{1/4} \|_{\infty} + \sum_{\mu \in \sigma_{\princ}} \| \Dpm{f} \|_{\infty} \right) + T^{\frac{1}{2}} (\log T)^{-1} \| \mathcal{R}f(\cdot, T) \|_{\infty} \\
& \leq 2C_M \|f\|_{W^6}  (\log T)^{-1},
\end{split}
\end{equation*}
and Lemma \ref{lemma:levy} allows once more to conclude.

In the last case, when $\mu_f > 1/4$, we have $\sigma_{\comp} = \emptyset$ and $\varepsilon_0=0$. 
Let $A$ be the random variable 
$$
A = T^{-\frac{1}{2}} \int_0^T f \circ h_t(x) \diff t,
$$
with $x \sim \vol$.
From Theorem \ref{thm:main1} we get the estimate
\begin{equation*}
\begin{split}
&\left\lvert A - \left( \sum_{\mu \in \sigma_{\princ}} \cos\left(\frac{\Im\nu}{2} \log T \right) \, \Dp{ f} + \sin\left(\frac{\Im\nu}{2} \log T \right) \, \Dm{ f} \right)  \circ \phi^X_{\log T} \right\rvert \leq T^{\frac{1}{2}}  \| \mathcal{R}f(\cdot, T) \|_{\infty} \\
& \qquad \qquad \leq C_M \|f\|_{W^6} T^{-\frac{1}{2}} (1+\log T).
\end{split}
\end{equation*}
Thus, the last part of Theorem \ref{thm:main_BF1} follows again from Lemma \ref{lemma:levy}.

\section{Appendix. A temporal limit theorem \\ By Emilio Corso}\label{sec:temporal_lt}

For any bounded measurable $f\colon M\to \R$, we denote the ergodic integral of $f$ up to time $T\in \R_{>0}$ along the horocycle orbit of a point $x\in M$ as 
\begin{equation*}
I_f(x,T)\coloneqq\int_0^{T}f\circ h_t(x)\;\text{d}t.
\end{equation*}

Let us recall that, for any $\sigma\in \R$, $\cN(0,\sigma^{2})$ indicates the Gaussian distribution on $\R$ with mean $0$ and variance $\sigma^{2}$, and, for any $T\in \R_{>0}$, $\mathcal{U}_{[0,T]}$ denotes the uniform probability measure on $[0,T]$.
We prove the following version of Theorem \ref{thm:main_DS}.

\begin{theorem}
\label{temporalDLT}
	Let $f \in W^6(M)$ be a real-valued function with $\vol(f)=0$. 
Assume that
\begin{itemize}
\item[(a)] either $1/4 \notin \Spec(\square)$ and $\Dpm{f} \equiv 0$ for all $\mu \in \Spec(\square) \cap \R_{>0}$,
\item[(b)] or all components $f_\mu$ of $f$ in \eqref{eq:decomp_f_irreps} corresponding to positive parameters $\mu >0$ are measurable coboundaries.
\end{itemize}
If $f$ is not a measurable coboundary for $U$, then there is a real number $\sigma>0$ such that, for every $x \in M$,  
	\begin{equation}
	\label{distrlimit}
	\frac{I_f(x,t)+\int_0^{\log{T}}Vf_0\circ\phi^{X}_s(x)\;\emph{d}s}{\sqrt{\log{T}}}\overset{T\to+\infty}{\longrightarrow}\cN(0,\sigma^{2}), \quad t\sim \mathcal{U}_{[0,T]}
	\end{equation}
	in distribution. Therefore, the ergodic integrals of $f$ satisfy a temporal ditributional limit theorem on any horocycle orbit.
\end{theorem}

The proof of Theorem~\ref{temporalDLT}, which follows the lines of the proof of~\cite[Thm.~5.1]{Dolgopyat-Sarig}, combines the asymptotic expasion of ergodic averages provided by Theorem~\ref{thm:main1} together with the Central Limit Theorem for ergodic integrals along geodesic orbits proven by Ratner in \cite{Rat:CLT}.

\smallskip
Lemma \ref{lemma:log_growth} yields the asymptotic expansion 
\begin{equation}
\label{asympexp}
I_f(x,t)=\int_0^{\log{t}}\bigl(Vf_0\circ \phi^{X}_s\circ h_t(x)-Vf_0\circ \phi_s^{X}(x) \bigr)\;\text{d}s +\mathcal{E}(x,t) \text{ for any }t>0,
\end{equation} 
where $\mathcal{E}(x,t) \leq C_M\norm{f}_{W^6}$ is uniformly bounded and hence doesn't affect the distributional limit of~\eqref{distrlimit}.

Observe that the dependence on $t$ of the integral in~\eqref{asympexp} occurs both in the starting point $h_t(x)$ of the geodesic orbit and in the upper bound $\log{t}$ of the domain of integration. The following lemma provides a first reduction, in that it removes the latter dependence.

\begin{lemma}
	Let $f$ be as in Theorem~\ref{temporalDLT}. If  
	\begin{equation*}
	\frac{\int_0^{\log{T}}\bigl(Vf_0\circ \phi^{X}_s\circ h_t(x)-Vf_0\circ \phi_s^{X}(x) \bigr)\;\emph{d}s+\int_0^{\log{T}}Vf_0\circ\phi^{X}_s(x)\;\emph{d}s}{\sqrt{\log{T}}}\overset{T\to+\infty}{\longrightarrow}\cN(0,\sigma^{2}), \quad t\sim \mathcal{U}_{[0,T]}
	\end{equation*}
	in distribution, then~\eqref{distrlimit} holds.
\end{lemma} 
\begin{proof}
From \eqref{asympexp}, and since $f_0 \in W^4(H_0) \subset \mathscr{C}^2(M)$, it is enough to show that 
$$
\left\lvert \int_0^{\log{T}}\bigl(Vf_0\circ \phi^{X}_s\circ h_t(x)-Vf_0\circ \phi_s^{X}(x) \bigr) \diff s - \int_0^{\log{t}}\bigl(Vf_0\circ \phi^{X}_s\circ h_t(x)-Vf_0\circ \phi_s^{X}(x) \bigr) \diff s \right \rvert \leq \norm{f_0}_{\mathscr{C}^2}.
$$
We rewrite the left hand-side above as
		   \begin{equation*}
		   \begin{split}
& \left\lvert \int_0^{\log{(T/t)}}\bigl(Vf_0\circ \phi^{X}_{s+ \log t}\circ h_t(x)-Vf_0\circ \phi_{s+\log t}^{X}(x) \bigr) \diff s \right \rvert \\
& \qquad \qquad = \left\lvert \int_0^{\log{(T/t)}}\bigl(Vf_0\circ h_{e^{-s}} - Vf_0 \bigr) \circ \phi^{X}_{s+ \log t}(x) \diff s \right \rvert \leq  \int_0^{\log{(T/t)}} \norm{Vf_0\circ h_{e^{-s}} - Vf_0}_{\infty} \diff s \\
& \qquad \qquad \leq \int_0^{\infty} e^{-s} \norm{UVf_0}_{\infty} \diff s \leq \norm{f_0}_{\mathscr{C}^2},
		   \end{split}
		   \end{equation*}
which proves the lemma.
\end{proof}

We are thus left with the study of the distributional limit of the random variables 
\begin{equation*}
t\mapsto \frac{\int_0^{\log{T}}Vf_0\circ \phi^{X}_s\circ h_t(x) \;\text{d}s}{\sqrt{\log{T}}}, \quad t\sim \mathcal{U}_{[0,T]}.
\end{equation*}

For later convenience, we shall interpret the integral in the numerator as an ergodic integral along the backward geodesic orbit of $\phi_{\log{T}}^{X}\circ h_t(x)=h_{t/T} \circ \phi^{X}_{\log{T}}(x)$, that is, 
\begin{equation*}
\int_0^{\log{T}}Vf_0\circ \phi^{X}_s\circ h_t(x) \;\text{d}s=\int_{-\log{T}}^{0}Vf_0\circ \phi_s^{X}(h_{t/T}(\phi^{X}_{\log{T}}(x)))\;\text{d}s.
\end{equation*}

From now on, the proof is an articulation of the argument outlined in~\cite{Dolgopyat-Sarig}.
Recall that, by Lemma \ref{lemma:log_growth}, $Vf_0$ is not a measurable coboundary for $X$. In view of the main result of~\cite{Rat:CLT}, we know that, for some $\sigma >0$, 
\begin{equation*}
\frac{\int_{-\log{T}}^{0}Vf_0\circ \phi^{X}_s(y)\;\text{d}s}{\sqrt{\log{T}}}\overset{T\to+\infty}{\longrightarrow}\cN(0,\sigma^{2})
\end{equation*}
in distribution, when $y$ is sampled according to the Haar measure $\vol$ on $M$. This value of $\sigma$, explicitly computable as in~\cite[Thm.~3.1]{Rat:CLT}, will be fixed until the end. Eagleson's theorem \cite{Eagleson} ensures that the same convergence in distribution takes place if $y$ is sampled according to any probability measure which is absolutely continuous with respect to $\vol$. In our case, for each $T>0$, the distribution of $y=h_{t/T}(\phi^{X}_{\log{T}}(x))$ is given by the uniform probability measure $\nu_{\log{T}}$ on the unit-length horocycle arc $\gamma_{\log{T}}=\{h_{t/T}(\phi^{X}_{\log{T}}(x)):0\leq t\leq T \}$, which is singular with respect to $\vol$. The rest of the argument is devoted to explicate how it is possible to replace $\nu_{\log{T}}$ by an appropriate thickening, which is absolutely continuous with respect to $\vol$, without altering the distributional limit.

\smallskip
For simplicity, we adopt the notation $I_{Vf_0}^{\phi}(y,T)=\int_{-T}^{0}Vf_0\circ \phi^{X}_s(y)\;\text{d}s$, for any $T>0$.
\begin{proposition}
	\label{thickeningtrick}
	For any strictly increasing sequence $(T_n)_{n\in \N}\in (\R_{>0})^{\N}$, there is a subsequence $(T_{n_k})_{k\in \N}$ such that 
	\begin{equation*}
	\frac{I^{\phi}_{Vf_0}(y,\log{T_{n_k}})}{\sqrt{\log{T_{n_k}}}}\overset{k\to\infty}{\longrightarrow}\cN(0,\sigma^{2}), \quad  y\sim \nu_{\log{T_{n_k}}},
	\end{equation*}
	in distribution.
\end{proposition} 

By virtue of the previous considerations, Theorem~\ref{temporalDLT} follows at once from Proposition~\ref{thickeningtrick}.

\begin{proof}
	The set of all non-empty, compact subsets of the compact space $M$ is a compact metric space for the Hausdorff distance\footnote{Recall that the Hausdorff distance is defined by 
		\begin{equation*}
		d_H(C,K)\coloneqq\inf\{\varepsilon >0:C\subset K_{\varepsilon} \text{ and }K\subset C_{\varepsilon}  \}
		\end{equation*}
		for any non-empty compact subsets $C,K\subset M$, where $A_{\varepsilon }$ denotes the closed $\varepsilon $-neighborhood of a set $A\subset M$ with respect to a fixed Riemannian distance on $M$.}. It follows that there is a subsequence $\gamma_{\log{T_{n_k}}}$ converging to a compact set $\mathcal{K}\subset M$. It is straightforward to check that $\mathcal{K}=\bar{\gamma}$ is the unit-length horocycle arc $\{h_{u}(x_{*}):0\leq u\leq 1 \}$, where $x_*=\lim_{k\to\infty}\phi^{X}_{\log{T_{n_k}}}(x)$. 
		We now thicken the arc $\bar{\gamma}$ in the directions of the geodesic and the unstable horocycle flow, so as to obtain a parallelepiped (compact, with non-empty interior)
		\begin{equation}
		\label{parallelepiped}
		P=\{h^{\mathrm{u}}_{r}\circ \phi^{X}_s (y):-1/2\leq r,s\leq 1/2,\;y\in \bar{\gamma}  \},
		\end{equation}
		and denote by $\bar{\nu}$ the normalized restriction of $\vol$ to $P$, which is clearly absolutely continuous with respect to $\vol$. We claim that the distribution of $I_{Vf_0}^{\phi}(z,\log{T_{n_k}})/\sqrt{\log{T_{n_k}}},\;z\sim \nu_{\log{T_{n_k}}}$, is uniformly close, for all $k$ sufficiently large and in the topology of weak convergence for Borel probability measures on $\R$, to the distribution of the same random variable when $z$ is sampled according to $\bar{\nu}$.  In light of the already mentioned Eagleson's theorem, this achieves the proof of the proposition.
		
		More precisely, let us fix a bounded Lipschitz-continuous function $\varphi\colon \R\to \R$; for any $k\in \N$, we aim to show that
		\begin{equation}
		\label{infinitesimal}
		\biggl|\int_{M}\varphi\circ (\log{T_{n_k}})^{-1/2}I^{\phi}_{Vf_0}(\cdot,\log{T_{n_k}})\;\text{d}\nu_{\log{T_{n_k}}}-\int_{\R}\varphi\;\text{d}\cN(0,\sigma^{2})\biggr|\overset{k\to\infty}{\longrightarrow}0.
		\end{equation}
		For notational simplicity, we let $\psi_k\colon M\to \R$ denote the function $y\mapsto (\log{T_{n_k}})^{-1/2}I_{Vf_0}^{\phi}(y,\log{T_{n_k}})$.
		By the triangle inequality, we can bound the quantity in~\eqref{infinitesimal} from the above by 
		\begin{equation*}
		\biggl|\int_{M}\varphi\circ \psi_k \;\text{d}\nu_{\log{T_{n_k}}}-\int_{M}\varphi\circ \psi_k\;\text{d}\bar{\nu} \biggr|
		+\biggl|\int_{M}\varphi\circ \psi_k\;\text{d}\bar{\nu} -\int_{\R}\varphi\;\text{d}\cN(0,\sigma^{2})\biggr|,
		\end{equation*}
		where we already argued that the second term of the sum converges to $0$ as $k\to\infty$. We may thus focus on the first term. Consider the thickening $P_k$ of $\gamma_{\log{T_{_k}}}$, defined as in~\eqref{parallelepiped}, and let $\bar{\nu}_k$ be the normalized restriction of $\vol$ to $P_k$.
		 We estimate, again via the triangle inequality, 
		 \begin{equation*}
		 \begin{split}
		 \biggl|\int_{M}\varphi\circ \psi_k\;\text{d}\nu_{\log{T_{n_k}}}-\int_{M}\varphi\circ \psi_k\;\text{d}\bar{\nu} \biggr|&\leq 
		 \biggl|\int_{M}\varphi\circ \psi_k\;\text{d}\nu_{\log{T_{n_k}}}-\int_{M}\varphi\circ \psi_k\;\text{d}\bar{\nu}_k \biggr|\\
&+
		 \biggl|\int_{M}\varphi\circ \psi_k\;\text{d}\bar{\nu}_k-\int_{M}\varphi\circ \psi_k\;\text{d}\bar{\nu} \biggr|.
		 \end{split}
		 \end{equation*}
		   Since there are constants $C_{\phi},C_{h}>0$ such that  $d(\phi^{X}_s(y),\phi^X_s(y)')\leq C_{\phi}e^{|s|}d(y,y')$ and $d(h^{\mathrm{u}}_ry,h^{\mathrm{u}}_ry')\leq C_{h}(1+|r|+r^{2})d(y,y')$ for any $y,y' \in M$ and any $s,r\in \R$, it follows easily that the $P_k$ converge to $P$ in the Hausdorff distance. This implies that 
		   \begin{equation*}
		   \begin{split}
		   \biggl|\int_{M}\varphi\circ \psi_k\;\text{d}\bar{\nu}_k&-\int_{M}\varphi\circ \psi_k\;\text{d}\bar{\nu} \biggr|
		   =\frac{1}{\vol{P}}\biggl|\frac{\vol{P}}{\vol{P_k}}\int_{P_k}\varphi\circ \psi_k \; \text{d}\vol-\int_{P}\varphi\circ \psi_k\;\text{d}\vol\biggr|\\
		   &=\frac{1}{\vol{P}}\biggl|\biggl(\frac{\vol{P}}{\vol{P_k}}-1\biggr)\int_{P_k}\varphi\circ \psi_k\;\text{d}\vol+\int_{P_k}\varphi\circ \psi_k\;\text{d}\vol-\int_{P}\varphi\circ \psi_k\;\text{d}\vol\biggr|\\
		   &\leq \frac{\norm{\varphi}_{\infty}}{\vol{P}}\biggl(\biggl(\frac{\vol{P}}{\vol{P_k}}-1\biggr)+\vol(P_k\bigtriangleup P)\biggr)\overset{k\to\infty}\longrightarrow 0.
		   \end{split}
		   \end{equation*}
		   In order to show that the difference $|\int_{M}\varphi\circ \psi_k\;\text{d}\nu_{\log{T_{n_k}}}-\int_{M}\varphi\circ \psi_k\;\text{d}\bar{\nu}_k |$ is infinitesimal as well, we start by applying Fubini's theorem:
		   \begin{equation}
		   \label{Fubini}
		   \int_{M}\varphi\circ \psi_k\;\text{d}\bar{\nu}_k=\int_{-1/2}^{1/2}\int_{-1/2}^{1/2}\int_{0}^{1}\varphi\circ\psi_k(h^{\mathrm{u}}_r(\phi^{X}_{s}(h_u(\phi^{X}_{\log{T_{n_k}}}x))))\;\text{d}u\;\text{d}r\;\text{d}s.
		   \end{equation}
		   Secondly, we compare the two quantities
		   \begin{equation*}
		   \int_{M}\varphi\circ \psi_k\;\text{d}\nu_{\log{T_{n_k}}}=\int_{0}^{1}\varphi\circ\psi_k(h_{u}(\phi^{X}_{\log{T_{n_k}}}x))\;\text{d}u\; \text{ and }\;\int_{0}^{1}\varphi\circ\psi_k(h^{\mathrm{u}}_r(\phi^{X}_{s}(h_u(\phi^{X}_{\log{T_{n_k}}}x))))\;\text{d}u
		   \end{equation*}
		   for each fixed $r,s\in (-1/2,1/2)$. To simplify notation, we let $p_u=p_u^{(k)}=h_u(\phi^{X}_{\log{T_{n_k}}}(x))$ for any $u\in [0,1]$ and any $k\in \N$. We have, for any $u\in [0,1]$,
		   \begin{equation*}
		   \begin{split}
		 |\psi_k(h^{\mathrm{u}}_r&(\phi^{X}_s(p_u)))-\psi_k(p_u)|=|\psi_k(\phi^{X}_s(h^{\mathrm{u}}_{re^{-s}}(p_u)))-\psi_k(p_u)|\\
		 &=(\log{T_{n_k}})^{-1/2}\biggl|\int_{-\log{T_{n_k}}}^{0}Vf_0\circ \phi^{X}_{t+s}(h^{\mathrm{u}}_{re^{-s}}(p_u))\;\text{d}t-\int_{-\log{T_{n_k}}}^{0}Vf_0\circ \phi_t^{X}(p_u)\;\text{d}t\biggr|\\
		 &\leq (\log{T_{n_k}})^{-1/2}\bigg(\biggl|\int_{-\log{T_{n_k}}}^{0}Vf_0\circ \phi^{X}_{t+s}(h^{\mathrm{u}}_{re^{-s}}(p_u))\;\text{d}t-\int_{-\log{T_{n_k}}}^{0}Vf_0\circ \phi^{X}_{t}(h^{\mathrm{u}}_{re^{-s}}(p_u))\;\text{d}t\biggr|\\
		 &+\biggl|\int_{-\log{T_{n_k}}}^{0}Vf_0\circ \phi^{X}_{t}(h^{\mathrm{u}}_{re^{-s}}(p_u))\;\text{d}t-\int_{-\log{T_{n_k}}}^{0}Vf_0\circ \phi_t^{X}(p_u)\;\text{d}t\biggr| \bigg).
		 \end{split}
		   \end{equation*}
		 For the first addend, we estimate
		 \begin{equation*}
		 \begin{split}
		 \biggl|\int_{-\log{T_{n_k}}}^{0}Vf_0\circ \phi^{X}_{t+s}(h^{\mathrm{u}}_{re^{-s}}(p_u))\;\text{d}t&-\int_{-\log{T_{n_k}}}^{0}Vf_0\circ \phi^{X}_{t}(h^{\mathrm{u}}_{re^{-s}}(p_u))\;\text{d}t\biggr|=\\
		 &
		 \biggl|\int_{0}^{s}Vf_0\circ \phi^{X}_{t+s}(h^{\mathrm{u}}_{re^{-s}}(p_u))\;\text{d}t-\int_{-\log{T_{n_k}}}^{-\log{T_{n_k}}+s}Vf_0\circ \phi^{X}_{t}(h^{\mathrm{u}}_{re^{-s}}(p_u))\;
		 \text{d}t\biggr|\\
		 &\leq 2|s|\norm{Vf_0}_{\infty}\leq \norm{Vf_0}_{\infty}.
		 \end{split}
		 \end{equation*}
		 As to the second addend, we exploit the fact that
		 \begin{equation*}d(\phi^{X}_t(h^{\mathrm{u}}_{re^{-s}}(p_u)),\phi^{X}_t(p_u))\leq C_{\phi}e^{t}d(h^{\mathrm{u}}_{re^{-s}}(p_u),p_u)\leq C_{\phi}e^{t}\frac{\sqrt{e}}{2},
		 \end{equation*}
		  and obtain a bound
		 \begin{equation*}
		 \begin{split}
		 \biggl|\int_{-\log{T_{n_k}}}^{0}Vf_0\circ \phi^{X}_{t}(h^{\mathrm{u}}_{re^{-s}}(p_u))\;\text{d}t-\int_{-\log{T_{n_k}}}^{0}Vf_0\circ \phi_t^{X}(p_u)\;\text{d}t\biggr|&\leq C_{\phi}\frac{\sqrt{e}}{2}\text{Lip}(Vf_0)\int_{-\log{T_{n_k}}}^{0}e^{t}\;\text{d}t\\
		 &\leq C_{\phi}\frac{\sqrt{e}}{2}\text{Lip}(Vf_0).
		 \end{split} 
		 \end{equation*}
		 Combining the two upper bounds yields
		 \begin{equation*}
		  |\psi_k(h^{\mathrm{u}}_r(\phi^{X}_s(p_u)))-\psi_k(p_u)|\leq (\log{T_{n_k}})^{-1/2}\biggl(\norm{Vf_0}_{\infty}+C_{\phi}\frac{\sqrt{e}}{2}\text{Lip}(Vf_0)\biggr)
		 \end{equation*}
		 for any $u\in [0,1],r,s\in [-1/2,1/2]$. Integrating over $u$ we get
		 \begin{equation*}
		 \biggl|\int_{0}^{1}\varphi\circ\psi_k(p_u^{(k)})\;\text{d}u-\int_{0}^{1}\varphi\circ\psi_k(h^{\mathrm{u}}_r(\phi^{X}_{s}(p_u^{(k)})))\;\text{d}u\biggr|\leq \frac{\text{Lip}(\varphi)}{\sqrt{\log{T_{n_k}}}}\biggl(\norm{Vf_0}_{\infty}+C_{\phi}\frac{\sqrt{e}}{2}\text{Lip}(Vf_0)\biggr);
		 \end{equation*}
		 finally, using~\eqref{Fubini}, we conclude
		 \begin{equation*}
		 \begin{split}
		 \biggl|\int_{M}\varphi\circ \psi_k\;\text{d}\nu_{\log{T_{n_k}}}-\int_{M}\varphi\circ \psi_k\;\text{d}\bar{\nu}_k \biggr|&\leq\biggl|\int_{-1/2}^{1/2}\int_{-1/2}^{1/2}\frac{\text{Lip}(\varphi)}{\sqrt{\log{T_{n_k}}}}\biggl(\norm{Vf_0}_{\infty}+C_{\phi}\frac{\sqrt{e}}{2}\text{Lip}(Vf_0)\biggr)\;\text{d}r\;\text{ds}\biggr|\\
		 &= \frac{\text{Lip}(\varphi)}{\sqrt{\log{T_{n_k}}}}\biggl(\norm{Vf_0}_{\infty}+C_{\phi}\frac{\sqrt{e}}{2}\text{Lip}(Vf_0)\biggr)\overset{k\to\infty}{\longrightarrow}0.
		 \end{split}
		 \end{equation*} 
		 This finishes the proof.
\end{proof}

\end{document}